\documentclass[reqno,12pt]{amsart}
\usepackage{geometry}
\usepackage{amsfonts,amsmath,amssymb,amsthm,mathtools}
\usepackage[foot]{amsaddr}
\usepackage[colorlinks=true,citecolor=blue,urlcolor=black,linkcolor=blue]{hyperref}
\usepackage[numbers,sort&compress]{natbib}
\usepackage{graphicx}
\usepackage{subcaption}
\usepackage{pgfplots,pgfplotstable}
\pgfplotsset{compat=1.18}
\usepackage{tikz}
\usetikzlibrary{arrows.meta,positioning,chains,fit,shapes,calc,decorations}
\newcommand{\err}{t}
\newcommand{\Err}{T}
\newcommand*{\infn}[1]{\lVert{#1}\rVert}
\newcommand{\BFgamma}{\boldsymbol{\gamma}}
\newcommand{\BFGamma}{\boldsymbol{\Gamma}}
\newcommand{\Plim}{P_0}
\newcommand{\RR}{\mathbb R}
\newcommand{\NN}{\mathbb N}
\newcommand{\PP}{\mathbb P}
\newcommand{\EE}{\mathbb E}
\newcommand{\ve}{\varepsilon}
\newcommand{\tV}{\widetilde{V}}
\newcommand{\ec}{\mathbf{e(c)}}
\newcommand{\co}{\mathbf{c}}
\newcommand{\mA}{\mathcal{A}}
\newcommand{\mB}{\mathcal{B}}
\newcommand{\mC}{\mathcal{C}}
\newcommand{\mD}{\mathcal{D}}
\newcommand{\mE}{\mathcal{E}}
\newcommand{\mF}{\mathcal{F}}
\newcommand{\mG}{\mathcal{G}}
\newcommand{\mS}{\mathcal{S}}
\newcommand{\mT}{\mathcal{T}}
\newcommand{\Acrit}{A_{\normalfont\text{crit}}}
\DeclareMathOperator{\CR}{CR}
\DeclareMathOperator{\mCR}{{\overline {\CR}}}
\DeclareMathOperator{\Elev}{Elev}
\DeclareMathOperator{\Var}{Var}

\theoremstyle{plain}
\newtheorem{theorem}{Theorem}
\newtheorem{corollary}{Corollary}
\newtheorem{proposition}{Proposition}
\newtheorem{lemma}{Lemma}

\theoremstyle{definition}
\newtheorem{remark}{Remark}
\newtheorem{definition}{Definition}
\newtheorem{example}{Example}
\newtheorem*{example*}{Example}

\title[Evolution models with vanishing mutations]{On the occupation measure of evolution models with vanishing mutations}

\author[M. Bena{\"i}m]{Michel Bena{\"i}m$^{\ast}$}
\thanks{}
\address{$^{\ast}$Institut de Math\'{e}matiques, Universit\'{e} de Neuch\^{a}tel, Switzerland} 
\email{ \href{mailto:michel.benaim@unine.ch}{\nolinkurl{michel.benaim@unine.ch}}}

\author[M. Bravo]{Mario Bravo$^{\dagger}$}
\thanks{}
\address{$^{\dagger}$Facultad de Administración y Economía, Universidad de Santiago de Chile, Chile} 
\email{ \href{mailto:mario.bravo.g@usach.cl}{\nolinkurl{mario.bravo.g@usach.cl}}}

\author[M. Faure]{Mathieu Faure$^{\ddagger}$}
\address{$^{\ddagger}$Aix-Marseille University, CNRS, AMSE, France}
\email{ \href{mailto:mathieu.faure@univ-amu.fr}{\nolinkurl{mathieu.faure@univ-amu.fr}}}

\begin{document}

\begin{abstract}
We study the almost sure convergence of the occupation measure of evolution models where mutation rates decrease over time. We show that if the mutation parameter vanishes at a controlled rate, then the empirical occupation measure converges almost surely to a specific invariant distribution of a limiting Markov chain. Our results are obtained through the analysis of a larger class of time-inhomogeneous Markov chains with finite state space, where the control on the mutation parameter is explained by the energy barrier of the limit process. Additionally, we derive an explicit convergence rate, explained also through the tree-optimality gap, that may be of independent interest.

\end{abstract}

\maketitle
\begin{small}
\end{small}

\section{Introduction}
Evolution models have been a successful way to explain why some outcomes may be more plausible than others in dynamic strategic situations. A widely adopted approach to tackle this issue relies on stochastic evolution dynamics—models where agents adjust their strategies based on past payoffs and are occasionally subject to random perturbations or mutations (\citep{KMR93,young1993evolution,kanRob95}). As an illustration, in the seminal model by \citep{KMR93}, a finite population of agents repeatedly plays a symmetric game. Agents revise their strategies using an imitation dynamics but with a small probability of making mistakes, interpreted as random mutations whose rate is constant across states, agents, and over time. So, the process induces an ergodic Markov chain over the distribution of actions in the population and an action profile is said to be asymptotically stable if it belongs to the support of the limiting invariant distribution as the mutation tends to zero. As a counterpart, \citep{BerLip96} show that, without precise assumptions on the mutation process, any invariant distribution can be selected by choosing mutation effects that vary across states.

In this vein, \citep{Ell00} introduces a general framework that captures several features of the interactions, also assuming constant noise level. For instance, when considering a game played by a population, this allows to include effects as correlations or mutations that depend on realized payoffs. Given that the limit process is not necessarily ergodic, the approach studies the stochastic interaction among the recurrent classes induced in the limit of zero mutations. Central to the analysis are the concepts of radius—the minimal resistance (i.e., cost in terms of mutations) required to escape a stable state—and coradius—the resistance required to reach it. These quantities allow for a comparison of the robustness of equilibria and help identify which states are likely to be stochastically stable. A related model, that cannot be cast in the framework of \citep{Ell00}, is the so-called logit-response model (\citep{blume1993statistical,alos2010logit}), where the mutation structure is taken as a logit function of payoffs. The reader is referred to \citep[Chapter 12]{sandholm2010population} for further discussions and references on this type of models.

If mutations are interpreted as the result of dynamical experimentation, it is natural to allow the rate to vary over time, as it explains learning behaviors where agents experiment less as they gain experience. To the best of our knowledge, there exist few papers that study this case. In the framework of dynamics of equilibrium selection in evolutionary game theory, \citep{chen2001convergence} examine how the rate at which mutation probabilities decrease influences the convergence to a risk-dominant Nash equilibrium in $2\times 2$ coordination games. They show that if mutation rates decline too rapidly, the system may fail to converge.  Using a formulation that is equivalent to  considering vanishing mutations, \citep{sandholm1998evolution} study population growth on evolutionary dynamics. They prove that when growth is at least logarithmic, the evolutionary process converges to a single equilibrium. \citep{Pak08} adapts the approach in \citep{Ell00} to consider evolution models where the mutation rate declines over time, while the notion of selection refers to the convergence in distribution of the induced inhomogeneous Markov chain. Roughly speaking, \citep{Pak08} shows that if the mutation parameter decreases sufficiently slowly, at a rate controlled by the coradius and a cost of transition between classes in the limit, then the induced time-inhomogeneous Markov chain is strongly ergodic. This complement a similar result by \citep{Rob98} which is only valid where the minimum coradius recurrent class in the limit is aperiodic (see Section~\ref{sec:related} for a precise discussion).

\subsection{Our Contributions}
In this paper, we study a broad class of inhomogeneous evolution models, derived from the general framework of \citep{Ell00}. Our focus is on the long-term behavior of the empirical occupation measure when the noise in the system gradually vanishes over time. Viewing the model as one generated by repeated dynamical interactions, we adopt a natural notion of selection: a state is stable if, as time goes to infinity, the empirical occupation measure assigns a positive weight to it. Our main result, Theorem~\ref{thm:main}, shows that almost sure convergence of the occupation measure occurs provided the mutation rate decreases slowly enough. The admissible rate of decay is determined by the energy barrier associated with the graph structure induced by the limiting process. In addition, we establish an explicit speed of convergence in the $L^1$ norm, where the rate is also explained through the tree-optimality gap of said graph. We also show that the energy barrier is always smaller than the coradius-like quantity previously used to analyze these models via convergence in distribution (cf. Theorem~\ref{thm:comparison}).

From a technical standpoint, our model boils down to a time-inhomogeneous Markov chain, derived from a family of Markov chain with rare transitions.  Interestingly, even though our approach is primarily inspired by evolution models, the general class of processes we study in this paper contains some inhomogeneous chains already studied in the literature where we can recover or improve some known results (cf.~Proposition~\ref{prop:cloez}).   To study the occupation measure of such nonstationary chains, our approach relies on several ideas, crucially depending on estimations on the spectral gap of the induced time-dependent kernels, adapted to the nonreversible case through additive symmetrization. Using this estimates, we borrow some ideas from the literature on stochastic approximation algorithms for differential inclusion to obtain almost sure convergence and to derive explicit convergence rates.

\subsection{Organization of the paper}
This paper is organized as follows. In Section~\ref{sec:models}, we introduce the general model of evolution with vanishing noise and state our main result, Theorem~\ref{thm:main}. Section~\ref{sec:examples} presents examples illustrating the scope of the theorem, in particular showing how it recovers and complements certain results for related inhomogeneous Markov chains studied in the literature. In Section~\ref{sec:related}, we briefly compare our framework with existing work on inhomogeneous models of evolution and clarify the connection between the energy barrier and standard notions in this area. Section~\ref{sec:proof} discusses the main elements of the proof of Theorem~\ref{thm:main}, while the omitted proofs are collected in the Appendix.

\section{Model and main result} \label{sec:models}
In this paper $\mS$ is a finite state space and we denote by $\Delta(\mS)$ the set of probability distributions over the finite set $\mS$, that is  
$$\Delta(\mS)= \left\{ z \in \RR^{\mS} \, :\, \sum \nolimits_{x \in \mS} z(x) =1, \, z \geq 0\right\}.$$

\subsection{Homogeneous models of evolution}\hfill

\begin{definition} A map $\co: \mS \times \mS \rightarrow [0,+ \infty]$ is said to be an \emph{admissible cost function} if, for any $(x,y) \in \mS \times \mS$, there exists $m \in \NN$ and $(x_0,x_1, x_2,...,x_m) \in \mS^{m+1}$ such that $x_0=x$, $x_m=y$ and  $\co(x_i,x_{i+1}) < + \infty, \; i = 0,\ldots,m-1$.
\end{definition}

 The following is the notion of a model of evolution we adopt in this paper.
\begin{definition}\label{def:hom_ev}  Let $\co:\mS \times \mS \rightarrow [0,+ \infty]$ be an admissible cost function and let $\overline \ve>0$. An (homogeneous) model of evolution associated to $\co$ is a family of Markov kernels $\left (P_{\ve}\right )_{\ve \in (0,\overline \ve]}$ on $\mS$ satisfying the following property: for each $x,y \in \mS$, there exists $k(x,y)>0$ and a Lipschitz function $h_{x,y}$ such that $h_{x,y}(0)=0$ and 
\begin{equation}\label{eq:Pe}
\begin{cases}
 \dfrac{P_{\ve}(x,y)}{\ve^{\co(x,y)}} = k(x,y)(1+h_{x,y}(\ve))>0, & \text{ if } \co(x,y) < +\infty,\\[0.15cm]
P_{\ve}(x,y)=0, & \text{ otherwise. }
\end{cases}
\end{equation} 
\end{definition}
 Observe that from \eqref{eq:Pe} we have that the limit matrix  $\Plim:=\lim_{\epsilon \rightarrow 0} P_{\epsilon}$ exists. Without loss of generality, we assume throughout this paper that $\overline \ve=1$, and thus we refer to a model of evolution simply by $(P_\ve)_\ve$. Additionally, since $P_{\ve}$ is irreducible, it admits a unique invariant probability $\pi_{\ve} \in \Delta(\mS)$. Recall that for a sufficiently regular family of irreducible kernels $(P_\ve)_\ve$ with corresponding invariant distributions $(\pi_\ve)_\ve$, there exists $\pi^* \in \Delta(\mS)$ such that $\lim_{\ve \rightarrow 0} \pi_{\ve} = \pi^*$; we adopt this notation throughout. Moreover, $\pi^*$ is invariant under $\Plim$, which may admit other invariant distributions, and its support minimizes a related potential on $\mS$. We refer the reader to \citep{FreWen84} for further discussions and related notions.

\begin{remark}\label{rem:technical}
In Definition~\ref{def:hom_ev}, we assume that $k(x,y)>0$ for all $x,y \in \mS$, and that $P_\ve(x,y)=0$ for all $\ve$ small whenever the cost is $+\infty$. Under mild additional conditions, one could also analyze the case where $P_\ve(x,y)=h_{x,y}(\ve)$ with $h_{x,y}(\ve)\to 0$ as $\ve \to 0$. Also, it is possible to treat the case in which the functions $h_{x,y}$ are Hölder continuous instead of Lipschitz. For the sake of simplicity, we do not pursue this generalization here. However, we do present in Section~\ref{sec:examples} an example involving a related inhomogeneous Markov chain that exhibits this type of structure.
\end{remark}

\vspace{-1.2ex}

Definition~\ref{def:hom_ev}\,  builds on \citep[Definition~1]{Ell00}, but differs from it in two important ways. On the one hand, we require that for every $\ve \in (0,1]$ the transition matrix $P_\ve$ be irreducible, without imposing aperiodicity. On the other hand, we strengthen the regularity assumptions on the dependence of transition probabilities on $\ve$ by assuming the existence of Lipschitz functions $h_{x,y}$ that capture their asymptotic behavior as $\ve \to 0$, similarly to \citep{Pak08}.

To contrast, in addition to ergodicity, \citep[Definition~1]{Ell00} assumes that for any pair $(x,y)$ with $\co(x,y)<+\infty$,
\[\lim_{\ve \rightarrow 0} \dfrac{P_{\ve}(x,y)}{\ve^{\co(x,y)}} = k(x,y)>0,\]
with $P_{\ve}(x,y)=0$ otherwise. The strengthening of these regularity assumptions is motivated by technical considerations: it ensures sufficient stability of the associated inhomogeneous dynamics as the mutation parameter vanishes, ruling out oscillatory behavior. Importantly, these assumptions remain compatible with the standard classes of evolutionary models, including those in which transition probabilities depend polynomially on the mutation rate, as in \citep{young1993evolution,KMR93} or \citep{kanRob95}.

\begin{remark}\label{rem:MCRT}
Given a model of evolution $(P_{\ve})_{\ve}$, setting $\beta := - \log (\ve)$ and $p_{\beta} := P_{e^{-\beta}}$, a simple manipulation shows that, for all $(x,y)\in \mS\times \mS$ such that $\co(x,y)< +\infty$, 
\begin{equation*}\label{eq:rare}
     \lim_{\beta \rightarrow + \infty} \frac{- \log p_{\beta}(x,y) }{\beta} = \co(x,y),
\end{equation*}
with the convention $\log(0)=-\infty$.
So, in the terminology of  \citep{FreWen84}, the family of homogeneous Markov chains $(p_{\beta})_{\beta}$ has rare transitions with rate function $\co$. So, the class of models considered in this paper can be seen as belonging to this family with extra regularity as $\ve$ goes to zero. We refer the reader to \citep{Cat99} for additional references and results.
\end{remark}

\subsection{Inhomogeneous models of evolution}
Consider a model of evolution $(P_\ve)_\ve$ together with a strictly positive vanishing sequence $(\ve_n)_n$. These two ingredients jointly define what we call an {\em inhomogeneous model of evolution}, made precise below.  As mentioned earlier, these models are designed to capture the idea that in repeated dynamical interactions, agents become progressively less likely to make mistakes over time.

\begin{definition}[Inhomogeneous model of evolution] 
Let $(\ve_n)_n$ be a vanishing sequence of positive real numbers.  
The {\em inhomogeneous model of evolution associated to $(P_{\ve})_{\ve}$ and mutation rate $(\ve_n)_n$} is the time-inhomogeneous Markov process $(X_n)_n$ on $\mS$, defined by $X_1 = x_1 \in \mS$, and for all $n \geq 1$,
\[
\PP \left(X_{n+1} = y \mid \, X_{n}\right) = P_{\ve_n}(X_n,y), 
\quad x,y \in \mS.
\]
\end{definition}

 Throughout this paper, we denote by $(v_n)_n$ the sequence of empirical occupation measures associated to $(X_n)_n$ defined as
\begin{equation*}
v_n :=\frac{1}{n} \sum_{i=1}^{n} \delta_{X_i}  \in \Delta(\mS), \quad \text{for all }\, n \geq 1,
\end{equation*}
where $\delta_x \in \Delta(\mS)$ stands for the Dirac's delta at $x \in \mS$.

\subsection{Energy barrier and tree-optimality gap}

Let us first recall the notion of energy barrier associated to an admissible cost function $\co$, through the definition of an appropriate potential function on the set $\mS$ (cf. \citep{FreWen84, Mic92, trouve1996rough, Cat99}).  

For an admissible cost function $\co$, we define the (quasi) potential function $V$ as follows. Given a directed graph on $\mS$, $g \in 2^{\mS \times \mS}$ is an \emph{$x$-tree} if:

\begin{itemize}   
\item[$(i)$] The set of successors of $x$ in $g$ is empty.
\item[$(ii)$] For any $x' \in \mS \setminus \{x\}$, $x'$ has exactly one successor in $g$.
\item[$(iii)$] There is no loops in $g$. 
\end{itemize}

\vspace{1ex}
We denote by $\co(g)$  the aggregate cost of $g$, i.e. $\co(g) = \sum_{e \in g} \co(e)$, and we call $\mG_x$ the set of $x$-trees which have finite cost.

\begin{definition}
Given an admissible cost function $\co$, the (quasi) potential function $V:\mS \to [0,\infty)$ is defined by
\[V(x) := \tV(x) - \min_{x' \in \mS} \tV(x'), \; \text{ where } \; \tV(x) := \min_{g \in \mG_x} \co(g), \quad  x \in \mS.\]
\end{definition}
 Observe that the potential $V$ is always finite and that $\min_{x \in \mS} V(x) = 0$. This is an extension of the notion of energy barrier to the case when the cost $\co$ is induced by a potential function, that is $U(x)+\co(x,y)=U(y)+\co(y,x)$ for some $U \in \RR^\mS$. This framework has been widely used to study, for instance, the well-known simulated annealing algorithm (\citep{HolStr88}).

From now on, we denote by $\Gamma_{x,y}$ the set of paths from $x$ to $y$, that is
\[\Gamma_{x,y}:= \left\{\gamma=(x_0,...,x_k) \in \mS^{k+1}: \; \, x_0=x, \, x_k=y, \,\text{for some} \,\, k\in \NN\right\}.\]
Note that we do not restrict the definition of the set $\Gamma_{x,y}$ to paths with finite cost. For any edge $e=(x_-,x_+)$, its potential is defined as
\begin{equation*} 
W(e) := \min \{V(x_-) + \co(x_-,x_+), V(x_+) + \co(x_+,x_-) \}.
\end{equation*}
 For all $x, y \in \mS$, we define the elevation from $x$ to $y$ by
\begin{equation*}
\Elev(x,y) := \min_{\gamma \in \Gamma_{x,y}} ; \max_{e \in \gamma} W(e).
\end{equation*}
\begin{definition} The energy barrier relative to $\co$ is the finite and nonnegative quantity
\[\ec: = \max_{x,y \in \mS}\, \left\{\Elev(x,y) - V(x) - V(y)\right\}.\]
\end{definition}
 The reader is referred to \citep{Cat99} for further developments and references regarding the energy barrier.
 
Let us define the set of all $x$-trees $\mG= \bigcup_{x \in \mS}\mG_x$, the set $\mG^0$ of optimal trees:
\begin{equation}\label{eq:g0}
  \mG^0:=\{g \in \mG\,:\, \co(g)=\min_{z \in \mS}\min_{g_z \in \mG_z} \co(g_z)\},  
\end{equation}
and $\mG^1:=\mG \setminus \mG^0$ its complement. Also let us denote by $\co_0$ the minimum cost among all trees, that is, 
$$\co_0:= \min_{z \in \mS}\min_{g_z \in \mG_z}\co(g_z).$$
 Finally, for an admissible cost function $\co$, an important quantity for our analysis is the cost-difference between the second best and the optimal tree:
 \begin{definition}
 The {\em tree-optimality gap} of $\co$ is defined as
     \begin{equation} \label{eq:theta}
   \theta:= \min_{g \in \mG^1} \co(g) - \co_0 >0,
\end{equation}
with the convention that $\theta=+\infty$ if $\mG^1=\emptyset$, , that is, if the admissible cost $c$ satisfies $\co(x,y) = c$ for all $x \neq y$ such that $\co(x,y) < +\infty$.
\end{definition}

\subsection*{Additional notation}
In what follows, we adopt the standard asymptotic notation for sequences. That is, for $(a_n)_n$ and $(b_n)_n$ two sequences of real numbers, we say that $a_n=O(b_n)$ (resp. $a_n=\Omega(b_n)$) if there exists a constant $C$ such that $|a_n|\leq C|b_n|$ (resp. $|a_n|\geq C|b_n|$) for sufficiently large $n$.   We use $\widetilde O(b_n)$ to hide poly logarithmic terms of $b_n$. Also, for $U \in \RR^{\mS}$, we let $\infn{U}= \max_{x} |U(x)|$ be the infinity norm of $U$ and for $Q \in \RR^{\mS \times \mS}$ we denote by $\infn{Q}=\max_{x}\sum_y|Q(x,y)|$ the induced infinity norm over the matrices. Sometimes, $\infn{U}$ will denote the $\ell_1$ norm of $U$ when working in $\RR^{\mS}$, viewed as the space of signed measures.
\subsection{Main result}
We consider the following class of mutation rates. 

\begin{definition}
We say that the mutation rate $(\ve_n)_n$  is \emph{vanishing at speed $A >0$} (or \emph{$A$-vanishing}) if $\ve_n \to 0$ and
\begin{equation*}
   \left \{ \begin{aligned}
        &\ve_n= \Omega \left (n^{-A}\right),\\
        &|\ve_{n+1}-\ve_n|= O\left (n^{-(1+A)}\right).
    \end{aligned}
    \right.
\end{equation*}
\end{definition}

Observe that if $(\ve_n)_n$ is $A$-vanishing, then, using a telescoping argument as well as the fact that $\epsilon_n$ goes to zero,  one can check that $\ve_n=O(n^{-A})$.

\vspace{2ex}
We can now present the main result of this paper.

\begin{theorem} \label{thm:main} Let $(X_n)_{n\geq 1}$ be an inhomogeneous model of evolution with admissible cost function ${\bf c}$, 
and $A$-vanishing mutation rate $(\ve_n)_n$ with $2A\ec < 1$. 
Then\\[-2ex]
\begin{itemize}\setlength\itemsep{0.7em}
    \item[$(a)$]  $\lim_{n \to \infty} v_n = \pi^* \quad \mbox{almost surely}$;
    \item[$(b)$] Define $\Acrit := \frac{1}{2\left(\ec + \min \{\theta,1\}\right)}$. Then
    \begin{equation} \label{eq:rate}
    \EE(\infn{v_n -\pi^*}) = \begin{cases}
\widetilde O\left(\dfrac{1}{n^{A \min \{\theta,1\}}}\right)  & \text{ if } \; 0< A \leq \Acrit,\\[0.2cm]
\widetilde O\left(\dfrac{1}{n^{1/2 -A \ec}}\right) & \text{ if } \; \Acrit < A < \frac{1}{2\ec}.
\end{cases}\end{equation}
\end{itemize}
\end{theorem}

\vspace{2ex}

Let us observe that, in this general framework, the best possible rate provided by our result is obtained when $A=\Acrit$, yielding
$$\EE(\infn{v_n -\pi^*})=\widetilde O \left (
n^{-\mu}\right), \quad \text{with} \quad \mu=\frac{\min\{\theta,1\}}{2(\ec+\min\{\theta,1\})}.$$

The next corollary identifies simple cases where the convergence rate from Theorem~\ref{thm:main} admits a simpler expression.

\vspace{3ex}
\begin{corollary}
    Under the assumptions of Theorem~\ref{thm:main}, the following holds.
    \begin{itemize}\setlength\itemsep{0.7em}
    \item[$(i)$] If $V \equiv 0$, then the logarithmic terms in \eqref{eq:rate} can be avoided.
    \item[$(ii)$] If the cost function $\co$ takes integer values, then $\theta \geq 1$, and $\Acrit = \frac{1}{2(\ec + 1)}$. Therefore, \eqref{eq:rate} can be written as
\begin{equation*} 
    \EE(\infn{v_n -\pi^*}) = \begin{cases}
\widetilde O\left(\dfrac{1}{n^{A}}\right)  & \text{ if } \; 0< A \leq \Acrit,\\[0.25cm]
\widetilde O\left(\dfrac{1}{n^{1/2 -A \ec}}\right) & \text{ if } \; \Acrit < A < \frac{1}{2\ec}.
\end{cases}\end{equation*}
\end{itemize}
\end{corollary}

\section{Examples} \label{sec:examples}
\begin{example} \label{ex:1} Let $N \geq 1$ and $\mS=\{-N,...,0,...,N\}$. For $a, b>0$, consider the homogeneous model of evolution $(P_\ve)_\ve$ with off-diagonal elements given by

\[
\begin{cases}
P_\ve(x,x+1)=P_\ve(-x,-x-1)=\ve^a, & \,\, x=0,\ldots, N-1,\\
P_\ve(x,x-1)=P_\ve(-x,-x+1)=\ve^b, & \,\, x=1,\ldots, N,\\
P_\ve(x,y)=0,& \text{otherwise},\\
\end{cases}
\]
where we assume for simplicity, that $P_\ve$ is a stochastic matrix for all $\ve$ sufficiently small. The admissible cost function $\co$ associated to $(P_\ve)_\ve$, depicted in Figure~\ref{fig:ex}, is then given by
\begin{itemize}\setlength\itemsep{0.7em}
\item $\co(x,x+1) = \co(-x,-x-1) = a$ \, for $x=0,...,N-1$. 
\item $\co(x+1,x) = \co(-x-1,-x) = b$ \, for $x=0,...,N-1$. 
\item $\co(x,x)=0$ for all $x \in \mS$.
\item $\co(\cdot,\cdot) = + \infty$ elsewhere.
\end{itemize} 
\begin{figure}[ht!] 
\centering
\begin{tikzpicture}[thick,scale=1.2,-,shorten >= 1pt,shorten <= 1pt,every node/.style={scale=0.9,inner sep=0pt,minimum size=1.1cm}]
\begin{scope}[start chain=going right, node distance=1.1cm]
\node[on chain,draw,circle] (-N)  {\small $-N$};
\node[on chain,draw, circle] (-x-1)  {\small$-x-1$};
 \node[on chain,draw, circle] (-x)  {\small$-x$};
\node[on chain,draw,circle] (0) {\small$0$};
\node[on chain,draw, circle] (x)  {\small$x$};
 \node[on chain,draw, circle] (x+1)  {\small$x+1$};
  \node[on chain,draw, circle] (N)  {\small$N$};
\end{scope}
\path (-N) -- node[auto=false]{\ldots} (-x-1);
\path (-x) -- node[auto=false]{\ldots} (0);
\path (0) -- node[auto=false]{\ldots} (x);
\path (x+1) -- node[auto=false]{\ldots} (N);
\path[every node/.style={font=\small},->] (-x-1) edge [bend left = 35]   node [fill=white] {\small $b$}  (-x) ;
\path[every node/.style={font=\small},->] (-x) edge [bend left]   node  [fill=white]  {\small $a$} (-x-1) ;
\path[every node/.style={font=\small},->] (x+1) edge [bend right = 35]   node [fill=white] {\small $b$}  (x) ;
\path[every node/.style={font=\small},->] (x) edge [bend right = 35]   node [fill=white] {\small $a$}  (x+1) ;
\path[every node/.style={font=\small},->] (-N) edge [out=120,in=60,looseness=7] node[fill=white] {\small $0$} (-N);
\path[every node/.style={font=\small},->] (-x-1) edge [out=120,in=60,looseness=7] node[fill=white] {\small $0$} (-x-1);
\path[every node/.style={font=\small},->] (-x) edge [out=120,in=60,looseness=7] node[fill=white] {\small $0$} (-x);
\path[every node/.style={font=\small},->] (0) edge [out=120,in=60,looseness=7] node[fill=white] {\small $0$} (0);
\path[every node/.style={font=\small},->] (x) edge [out=120,in=60,looseness=7] node[fill=white] {\small $0$} (x);
\path[every node/.style={font=\small},->] (x+1) edge [out=120,in=60,looseness=7] node[fill=white] {\small $0$} (x+1);
\path[every node/.style={font=\small},->] (N) edge [out=120,in=60,looseness=7] node[fill=white] {\small $0$} (N);
\end{tikzpicture}
\caption{\label{fig:ex}
 The admissible cost function associated to Example~\ref{ex:1}}
\end{figure}
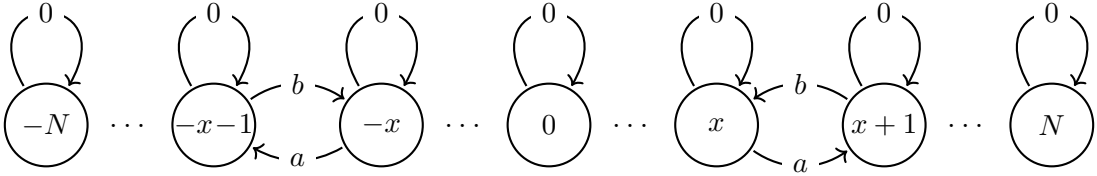

\noindent In this case, the limit matrix $\Plim$ is the identity and thus any distribution is invariant for $\Plim$. For this model, the energy barrier is given explicitly by (see Appendix~\ref{app:energy_example} for the computations)
\begin{equation} \label{eq:energy_ex1}
\ec=\begin{cases}
\,\,b & \text{ if } a \geq b, \\
(b-a)N+a  & \text{ if } a<b. \\
\end{cases}
\end{equation}

For $\ve>0$, the unique invariant probability distribution of the matrix $P_\ve$ is symmetric with respect to zero where, up to normalization,
\begin{equation*}
    \pi_\ve(0)=\ve ^{N(b-a)},\quad \text{and}\quad  \pi_\ve(k)= \ve^{(N-k)(b-a)},
\end{equation*}
for  $k=1,\ldots, N$. Hence
\[
\lim_{\ve \to 0} \pi_\ve=\pi^* :=
\begin{cases}
 \left (0,\ldots,0, 1,0,\ldots, 0 \right ) & \text{if} \quad a>b,\\[1ex]
     \dfrac{1}{2N+1}\left (1,\ldots,1 \right ) & \text{if} \quad a=b,\\[1ex]
    \,\,\left (1/2,0,\ldots, 0,1/2 \right ) & \text{if} \quad a<b.\\
\end{cases}
\]
\vspace{2ex}

Let us compute the tree-optimality gap $\theta$ by noting that there is a unique $x$-tree for every $x \in \mS$.

\noindent \textsc{Case} $a > b$.
 The set $\mG^0$ only contains the $0$-tree, and $\co_0 = 2Nb$. The set $\mG^1$ contains the unique $x$-tree, for $x \neq 0$. The least costly are the $1$-tree and the $-1$-tree, with associated cost $(2N-1)b +a$. Hence $\theta = a -b$.

\noindent\textsc{Case} $a < b$. The set $\mG^0$ only contains the $N$-tree and the $-N$-tree, and $\co_0 = N(a+b)$. The set $\mG^1$ contains the unique $x$-tree, for $x \neq -N,N$. The least costly elements of $\mG^1$ are the $(N-1)$-tree and the $(-N+1)$-tree, with associated cost $(N+1)b +(N-1)a$. Hence $\theta = b -a$. 

\noindent \textsc{Case} $a = b$. In that case, every tree has the same cost, so $\mG^1 = \emptyset$ and $\theta=+\infty$. 

Let us assume, for the sake of the example, that $a=1>b $ thus $\ec=b$ and $\theta=1-b$. Taking $\ve_n=n^{-A}$, we have that for all $A<(2b)^{-1}$, $v_n$ converges almost surely to $\delta_{0}$ and the best rate of convergence implied by Theorem~\ref{thm:main} is $\widetilde O(n^{-(1-b)/2})$, which  worsens arbitrarily as $b$ approaches one. On the other hand, when $b=1$ and setting $A=\Acrit=1/4$, $v_n$ goes to the uniform distribution with a rate of $\widetilde O(n^{-1/4})$.
\end{example}

\begin{example}\label{ex:pak}
This example, taken from \citep{Pak08}, serves both as a point of comparison with the literature on inhomogeneous models (see Section~\ref{sec:models}) and as an illustration of a case where the limiting kernel $\Plim$ contains periodic classes. Two players play the two-action coordination game presented in Figure~\ref{fig:pak_game} below, repeated in discrete time. Actions for players are denoted $\{\text{T},\text{B}\}$ and $\{\text{L},\text{R}\}$ for Player 1 and Player 2, respectively.

The model introduces experimentation to the purely myopic Best-Response in the following manner. For $0<\ve\leq1$, if a player receives a larger or equal payoff than her opponent at some time, then she experiments with probability $\ve^2$ (by using the other action). On the contrary if she receives less, then she applies experimentation with probability $\ve$. The evolution of the action profile defines an ergodic Markov chain on the set of pure strategy profiles $\mS= \{\text{TL}, \text{TR}, \text{BL}, \text{BR} \}$ whose transition matrix, up to row normalization and negligible terms, is given by
\[P_\ve=
\begin{pmatrix}
1 & \ve & \ve^2& \ve^3\\
\ve^2 &\ve^4 & 1 & \ve^2 \\
\ve^2 &1 &\ve^4 &\ve^2  \\
\ve^3 &\ve^2&\ve& 1
\end{pmatrix}, \quad \text{with}\quad\Plim=
\begin{pmatrix}\,1 \, & \,0 \, & \,0 \, & \,0 \,\\
\,0 \, &\,0 \, & \,1 \, & \,0 \, \\
\,0 \, &\,1 \, &\,0 \, & \,0 \, \\
\,0 \, &\,0 \, &\,0 \,& \,1 \, 
\end{pmatrix}.
\]

The induced graph with the admissible cost $\co$ is presented in Figure~\ref{fig:pak_graph}.

 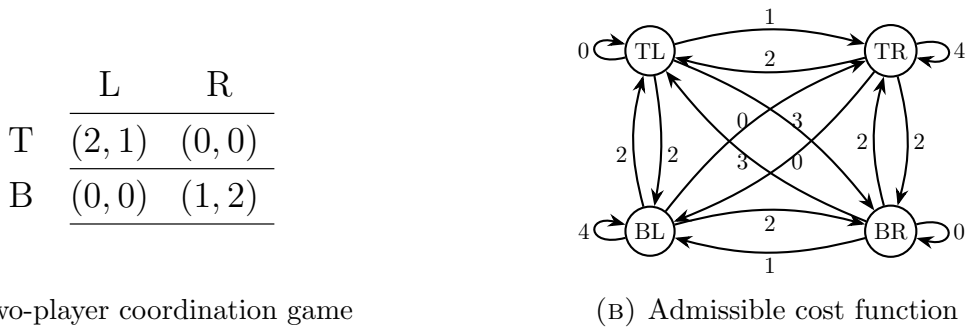
\begin{figure}[ht!]
\centering
\begin{subfigure}{0.4\textwidth}
\begin{minipage}[t]{\linewidth}
\centering
\scalebox{1}{%
\renewcommand{\arraystretch}{1.3}
\begin{tabular}{c@{\hspace{1em}}cc}
& L & R \\
\cline{2-3}
T & $(2, 1)$ & $(0, 0)$ \\
\cline{2-3}
B & $(0, 0)$ & $(2, 1)$ \\
\cline{2-3}\\
\end{tabular}
}
\end{minipage}
\caption{Two-player coordination game}
\label{fig:pak_game}
\end{subfigure}
\hfill
\begin{subfigure}{0.58\textwidth}
\centering
\begin{minipage}[t]{\linewidth}
\centering
\scalebox{1}{%
\begin{tikzpicture}[
  ->,
  >=Stealth,
  thick,
  scale=0.9,
  transform shape,
  inner sep=2.5pt,
  font=\footnotesize
]
  \node[circle, draw] (TL) at (0,0) {\small TL};
  \node[circle, draw] (TR) at (4,0) {\small TR};
  \node[circle, draw] (BL) at (0,-3) {\small BL};
  \node[circle, draw] (BR) at (4,-3) {\small BR};

  \draw[->] (TL) to[bend left=15] node[above] {\small $1$} (TR);
  \draw[->] (TL) to[bend left=10] node[below right] {\small $2$} (BL);
  \draw[->] (TL) to[bend left=15] node[right] {\small $3$} (BR);
  \draw[->] (TL) to[loop left] node[left] {\small $0$} (TL);

  \draw[->] (TR) to[bend left=15] node[above] {\small $2$} (TL);
  \draw[->] (TR) to[loop right] node[right] {\small $ 4$} (TR);
  \draw[->] (TR) to[bend left=15] node[right] {\small $ 0$} (BL);
  \draw[->] (TR) to[bend left=15] node[right] {\small $ 2$} (BR);

  \draw[->] (BL) to[bend left=15] node[below] {\small $ 2$} (BR);
  \draw[->] (BL) to[loop left] node[left] {\small $ 4$} (BL);
  \draw[->] (BL) to[bend left=15] node[left] {\small $ 0$} (TR);
  \draw[->] (BL) to[bend left=15] node[below left] {\small $ 2$} (TL);

  \draw[->] (BR) to[bend left=15] node[left] {\small $ 3$} (TL);
  \draw[->] (BR) to[bend left=15] node[left] {\small $ 2$} (TR);
  \draw[->] (BR) to[bend left=15] node[below] {\small $ 1$} (BL);
  \draw[->] (BR) to[loop right] node[right] {\small $ 0$} (BR);
\end{tikzpicture}
}
\end{minipage}
\caption{Admissible cost function}
\label{fig:pak_graph}
\end{subfigure}
\caption{Game and cost for Example~\ref{ex:pak}}
\end{figure}

Notice that $\Plim$ contains the periodic recurrent class $\{\text{TR},\text{BL} \}$ and that the invariant distribution of $P_\ve$ is, up to normalization, 
\[
\pi_{\ve}=\begin{pmatrix}\ve, & (1+ \ve^2)/2, &(1+ \ve^2)/2, &   \ve \end{pmatrix}, \quad \text{with} \quad \lim_{\ve \to 0}\pi_{\ve}= \pi^*=\left (0,\, 1/2,\, 1/2,\, 0 \right ).
\]

Consequently, $\pi^*$ puts weight exclusively on action profiles with zero payoff, explained by the fact that players experiment at a higher rates when actions mismatch. 

In this case, routine computations show that $\tV(\text{TL}) = \tV(\text{BR})=3$ and $\tV(\text{TR})=\tV(\text{BL})=2$, so that $V(\text{TL})=V(\text{BR})=1$, $V(\text{BL})=V(\text{TR})=0$. Thus, the energy barrier is $\ec =\Elev(\text{TL},\text{TR}) - V(\text{TL}) - V(\text{TR})= 1$, $\co_0=2$, and $\theta=1$. Taking the mutation parameter as $\ve_n= 
n^{-A}$, Theorem~\ref{thm:main} states that the occupation measure $v_n$ converges almost surely to $\pi^*$ for all $A<1/2$ and that it does so with the rate $ \EE(\infn{v_n -\pi^*})=\widetilde O( n^{-\min\{1/2-A,A\}}).$
\end{example}

\begin{example}\label{ex:Dietz} For $N \geq 2$, consider the simple model of evolution on $\mS=\{1,...,N\}$, given by
\[
P_\ve = \begin{cases}1- \ve\sum_{y \neq x}k(x,y) & \text{if  } x=y\\[0.5ex] \,\,\,\,\ve\, k(x,y) & \text{otherwise,} \end{cases}
\]
where for all $x \neq y$, $k(x,y)$ is a strictly positive number. The admissible cost $\co$ is $\co(x,y)=1$ if $x \neq y$ and $\co(x,x)=0$ and we assume for simplicity that $P_\ve$ is a stochastic matrix for all $0<\ve \leq 1$. Observe that in this case $\Plim=I$ and let us take again $\ve_n=n^{-A}$ for some $A>0$.  In this case, we have that for all $x \in \mS$ every $x$-tree in $\mG_x$ has cost $N-1$, so $\tV(x)=N-1$ and then $V \equiv 0$ and $\theta=+\infty$. Furthermore, $W(e)=1$ for all arc $e$ and therefore $\ec=1$. Thus, our result prescribes almost sure convergence of the occupation measure towards $\pi^*$ for $A<1/2$, independent of $N$, and that $\EE(\infn{v_n-\pi^*})=O(n^{-(1/2-A)})$.
 \end{example}

The inhomogeneous Markov chain induced by the model in Example~\ref{ex:Dietz} above, when $\ve_n$ is exactly $n^{-A}$, corresponds to the one studied in \citep{DieSet07}. The authors characterize (almost) completely the asymptotic behavior of the occupation measure $v_n$. The main result states that if $A=1$, then $v_n$ converges in distribution to a measure with full support on $\Delta(\mS)$, with explicit characterization of said distributions in some cases. It is worth noting that \citep{gantert} was the first to prove that when $A=1$, $v_n$ cannot converge in probability to a constant vector. If $A >1$, $v_n$ converges in distribution to a vector that depends on the initial condition. When $0<A<1/2$, $v_n \to \pi^*$ almost surely and the convergence is in probability for $1/2\leq A<1$. The authors conjecture, relying on numerical experiments,  that the latter might also hold in the almost sure sense,  noting that the rate degenerates when $A \sim 1$; the question was left open. A proof in the case $N=2$ was obtained later in \citep{EngVol18}. For $N>2$, the question was answered in the positive in \citep{BourCloez18}, while studying a related class of inhomogeneous Markov chains (see Example~\ref{ex:cloez} for more details).

Interestingly, given the particular structure in the previous example, a minor adjustment of the argument to obtain Theorem~\ref{thm:main} allows us to give an alternative proof for convergence when $0<A<1$ with rate $O(1/n^{(1-A)/2})$, explaining how the convergence rate deteriorates as $A$ approaches one (see Proposition~\ref{prop:cloez}). The mentioned adjustment relies on two simple facts. First, the invariant distribution of $P_{\varepsilon_n}$ does not depend on $n$. Second, as the model suggest, $P_\ve$ approaches the identity as $\infn{I-P_\ve}=O(\ve)$, thus the bounds in our proof can be a lot more precise.

The following inhomogeneous Markov chain, studied in \citep{BourCloez18}, provides an example that does not strictly fit into the framework of an inhomogeneous model of evolution. Nevertheless, it can still be treated within our analysis (cf. Remark~\ref{rem:technical}).  

\vspace{2ex}
\begin{example}\label{ex:cloez}
Consider the inhomogeneous Markov chain with transitions
\[
\widetilde P_n(x,y)= \ve_n(k(x,y) + \err_n(x,y)),\,\, x\neq y.
\]
For $x \neq y$, we assume that $k(x,y) \geq 0$, that $\err_n(x,y) \to 0$, and that $(\ve_n)n$ is a strictly positive, decreasing sequence converging to some $\ve \in [0,1]$, with $\sum_{n \geq 1} \ve_n = +\infty$. Also, $I+k$ is stochastic and irreducible, where $k (x,x):=-\sum_{x\neq y}k(x,y)$, so that it admits a unique invariant distribution $\pi^*$. Observe that $\widetilde \Plim:=\lim_{n\to +\infty} \widetilde P_n= I+\ve k$, and that $\pi^*$ is also an invariant distribution for $\widetilde\Plim$, which is unique whenever $\ve >0$. Relying on a comparison with a suitable continuous-time Markov process, \citep{BourCloez18} obtain, among other results, that when $\sum_{n \geq 1}(\ve_n n^2)^{-1}<+\infty$, then $v_n$ converges almost surely to $\pi^*$, and an almost sure convergence rate for $\infn{v_n -\pi^*}$ is established particularly in the case where $|\err_n(x,y)|= O(n^{-B})$, $0<B\leq 1$. As an illustration, if $\ve=0$ and we take $\ve_n=n^{-A}$, \citep[Theorem 2.12]{BourCloez18} shows that almost sure convergence holds for $0<A<1$. Note that, when $k(x,y)>0$ for all $x \neq y$, $\ve=0$, and $\err_n\equiv 0$, this proves the conjecture discussed in Example~\ref{ex:Dietz}.
\end{example}

The following proposition, whose proof is presented in Appendix~\ref{app:examples}, shows how the previous results can be recovered within the framework of our analysis. For the sake of simplicity, we only treat the case $\ve = 0$ in Example~\ref{ex:cloez}.
\begin{proposition}\label{prop:cloez}
    Consider the inhomogeneous Markov chain in Example~\ref{ex:cloez} with $\ve=0$. Let $\err_n:=\max_{x\neq y}|\err_n(x,y)|$, $\Err_n:=\frac{1}{n}\sum_{i=1}^n \err_i$, and assume that $(\ve_n)_n$ is $A$-vanishing for $0<A<1$. Then $v_n$ converges almost surely to $\pi^*$ and, moreover,
        \begin{equation*}
    \EE(\infn{v_n -\pi^*}) =  O\left ( n^{-(1-A)/2} +  \Err_n\right ).    
\end{equation*}
 In particular:
 \begin{itemize}\setlength\itemsep{0.6em}
     \item[$(a)$] If $\sum_{n\geq 1}\err_n <\infty$, then  $\EE(\infn{v_n -\pi^*})= O  (n^{-(1-A)/2}  )$, as in the case of Example~\ref{ex:Dietz}.
     \item[$(b)$] If $\sum_{n\geq 1}\err_n =+\infty$ with $\err_n= O( n^{-B})$, $0<B\leq 1$, then $$\EE(\infn{v_n -\pi^*}) =  O\left ( n^{-\mu}\right ),\quad \text{where}\quad  \mu=\min \left\{(1-A)/2,B\right\}.$$
 \end{itemize}
\end{proposition}

\begin{remark}\label{rem:ex1}
In Example~\ref{ex:1}, it is also possible to show that convergence holds for a wider range of mutation rates. In that case,  $\infn{I-P_\ve}=O(\ve^{\min\{a,b\}})$ and consequently the almost sure convergence result can be strengthened to $A$-vanishing rates where $A(2\ec -\min\{a,b\}) <1$. For instance, when $0<b<a$, $v_n \to \pi^*$ for all $0<A< b^{-1}$ instead of $0<A< (2b)^{-1}$. More generally, this type of improved result can be obtained   for any model with $\Plim=I$ since in that case $\infn{I-P_\ve}=O(\ve^{ \underline{c}})$, where $\underline{c}$ is the smallest (finite) nonzero cost.
\end{remark}
\section{Related work on inhomogeneous models} \label{sec:related}

\subsection{Coradius and energy barrier} Let $\co: \mS \times \mS \rightarrow [0,+\infty]$ be an admissible cost function. Given $\mA,\mB$ two subsets of $\mS$, $\Gamma_{\mA,\mB}$ denotes the set of paths from $\mA$ to $\mB$, and define the resistance from $x$ to $\mA$ as the minimum cost to travel from $x$ to $\mA$,  i.e.
\[r(x,\mA) := \min \{\co(\gamma): \; \, \gamma \in \Gamma_{x,\mA}\} < + \infty.\]
The coradius of a set $\mA$ is defined as the maximal resistance from $\mS\setminus \mA$ to $\mA$, that is
\[\CR(\mA) := \max_{x \in \mS\setminus \mA} r(x,\mA).\]
\indent Given a model $(P_\ve)_{\ve}$, consider the time homogeneous Markov chain with transition matrix $\Plim=\lim_{\ve \to 0}P_\ve$ and let $\mT  \cup \mC_1 \cup \cdots \cup \mC_m =\mS$
be the standard decomposition of the state space $\mS$, where $\mT$ is the set of transient states and $\mC_i$, $i= 1,\ldots, m$, are the recurrent classes. Observe that for any homogeneous model, $\CR(x)= \CR(\mC_i)$ for all states belonging to the recurrent class $\mC_i$. Denote $\mC:= \cup_{i=1}^m\mC_i$ and define $\mCR(\co)$ as the minimum coradius over $\mC$

\[\mCR(\co):= \min_{x \in \mC} \CR(x).\]

 The quantity $\mCR(\co)$ has been widely used to study stochastic stability results for homogeneous evolution models. For instance, \citep[Theorem 1]{Ell00} shows that if a set $\mD\subseteq \mS$ is such that $\CR(\mD)<\CR(\mS\setminus \mD)$, then the set of stochastically stable states is contained in $\mD$. Also, for every state $z \notin \mD$, a bound of $O(\ve^{-\mCR(\co)})$ is established for the expected time to hit $\mD$. In fact, this can be improved considering a modified version of the coradius (see \citep[Theorem 2]{Ell00}).

\begin{remark}
In Example~\ref{ex:1}, we have $\CR(x) = bN + a|x|$ for all  $x \in \mS$,
and therefore $\mCR(\co)=\CR(0) = b N$, so $\ec\leq \mCR(\co)$ (see \eqref{eq:energy_ex1}) with equality if and only if $a \geq b$ and $N=1$. On the other hand, if $a \geq b$ and $N$ is large, one has that $\ec=b << \mCR({\co})$.  
\end{remark}

The following result, whose proof can be found in Appendix~\ref{app:thm:comparison}, shows that the comparison between $\mCR(\co)$ and  $\mathbf{e(c)}$ holds for any evolution model given by Definition~\ref{def:hom_ev}.

\begin{theorem}\label{thm:comparison} Consider a homogeneous model of evolution $(P_\ve)_{\ve \in (0,1]}$ with admissible cost function $
\co:\mS\times \mS \to [0,\infty]$. Let $\mC$ be the union of the recurrent classes induced by the limit kernel $\Plim$. Then $\ec \leq \CR(z)$ for all $z \in \mC$. Consequently $\ec \leq \mCR(\co)$.
\end{theorem}

\vspace{1ex}
\subsection{Ergodicity, corradius and inhomogeneous models}
Let $(X_n)_{n\geq 1}$ be a time inhomogeneous Markov chain in $\mS$ with transition matrices $(P_n)_{n\geq 1}$. Also, for all $1\leq m<n$, let $P^{(m,n)}:= P_m \cdot P_{m+1}\cdots P_{n-1}$ be the $(n-m)$-step transition starting from time $m$. Recall that $(X_n)_{n}$ is
 weakly ergodic if, for all $m \geq 1$ and $x,y,z \in \mS$,
\begin{equation}
\label{eq:w_ergodic}
\lim_{n \rightarrow + \infty} \left| P^{(m,n)}(x,z) - P^{(m,n)}(y,z) \right| = 0.
\end{equation} 
This property is also known as total variation merging; see \citep{saloff2009merging,saloff2011merging} for related notions and conditions ensuring quantitative convergence rates in \eqref{eq:w_ergodic}. 

The chain is strongly ergodic if there exists a probability distribution $\pi\in \Delta(\mS)$ such that
$$\lim_{n \rightarrow + \infty}  P^{(m,n)}(x,z)  = \pi(z),\quad x,z \in \mS.$$

As mentioned, the inhomogeneous model of evolution in \citep{Pak08} is the closest to our approach, where the underlying homogeneous model is more restrictive than the one given by Definition~\ref{def:hom_ev}. Precisely, the following additional hypotheses to Definition~\ref{def:hom_ev} are considered.

\begin{equation}\label{hyp:pak}
\left\{
\begin{aligned}
& \bullet \,\, \text{The cost function}\,\,\co \,\,\text{takes values in } \{0,1,2,\ldots \} \cup \{+\infty\},\\
&\bullet\, \text{The transition matrix }  P_{\ve}\text{ is ergodic for all }  \ve \in (0,1].
\end{aligned}
\right.
\end{equation}

One of the earliest inhomogeneous models of evolution is due to \citep{Rob98}, which corresponds to the framework introduced in \citep{KMR93}. This model satisfies \eqref{hyp:pak} and a decreasing vanishing mutation rate is considered. Roughly speaking, it is shown that if $\ve_n=n^{-A}$ with $A$ such that $A\mCR(\co)\leq1$, then the induced inhomogeneous Markov chain is strongly ergodic, with limiting distribution $\pi^*$. This result applies only to the case where the class achieving $\mCR(\co)$ is aperiodic (cf. Examples~\ref{ex:pak} and \ref{ex:pak_cont}).

 In order to take into account possible periodicity of the chain induced by the transition matrix $\Plim$, and using the fact that $P_\ve$ is ergodic for all $\ve \in (0,1]$, \citep{Pak08} introduces the quantity
\begin{equation*}
\eta(x):=\inf_{0<\ve\leq 1}\min_{m \geq 1} \max_{y \in \mC_x}r_m^\ve(y,x), \quad   x \in \mC,
\end{equation*}

\noindent where $r_m^\ve(y,x)$ is the resistance
between $y$ and $x$ in exactly $m$ steps for the ergodic transition matrix $P_\ve$, and $\mC_x$ is the recurrent class (for the limit matrix $\Plim$) to which $x$ belongs. Notice that if $\mC_x$ is an aperiodic recurrent class, then $\eta(x)=0$. Let us define 
    \[
\nu(\co):= \min_{x \in \mC} \CR(x) + \eta(x).
\]
\begin{remark}
Notice that Theorem~\ref{thm:comparison} implies that, for all admissible cost function $\co$, one has that $\ec \leq \nu(\co)$.
\end{remark}

The following statement summarizes the main results established in \citep{Pak08}, relying mainly on the application of the classical characterization of weak ergodicity by means of ergodic coefficients (see for instance \citep{MadIsa73, Ios69,bremaud2013markov}).

\begin{theorem}[\citep{Pak08}]\label{thm:pak}
Consider a homogeneous model of evolution $(P_\ve)_{\ve}$ satisfying \eqref{hyp:pak}, and assume that the sequence $(\ve_n)_n$ is such that
$\ve_n = \Omega\bigl(n^{-A}\bigr)$ with $0<A\nu(\co)\leq 1$.
Then the induced inhomogeneous model is weakly ergodic. If, in addition, $\sum_{n \geq 1} \lvert \ve_{n+1} - \ve_n \rvert < \infty$, 
then strong ergodicity holds with limiting distribution $\pi^*$.
\end{theorem}

\begin{example}[Example~\ref{ex:pak} continued]\label{ex:pak_cont} As mentioned, Example~\ref{ex:pak} shows the importance of the term $\eta(x)$ to establish strong ergodicity by means of Theorem~\ref{thm:pak}. In this example, $1= \CR(\{\text{BL}\})=\CR(\{\text{TR}\})$ and $\CR(\{\text{TL}\})=\CR(\{\text{BR}\})= 3$, so that $\mCR(\co)=1$. Given that $\{\text{TL}\}$ and $\{\text{BR}\}$ are aperiodic classes, $\eta(\text{TL})=\eta(\text{BR})= 0$. Also, one has that  $\eta(\text{BL})=\eta(\text{TR})= 3$, then $\nu(\co)=3$ and Theorem~\ref{thm:pak} implies strongly ergodic if $\ve_n=\Omega(n^{-A})$ with $A\leq 1/3$. In fact, it is proven for this specific example that weak ergodicity does not hold when  $1/2<A\leq 1$, showing that the aforementioned result of \citep{Rob98} does not apply when the minimum coradius class is periodic.
   \end{example}

\section{Proof of Theorem \ref{thm:main}} \label{sec:proof}
\subsection{Overview of the proof}

Using classical tree machinery for Markov chains, we first derive an explicit formula for the limit distribution $\pi^*$ associated with a model $(P_\ve)_\ve$, and show that $\infn{\pi_\ve - \pi^*}$ converges to zero at a rate controlled by the tree-optimality gap~$\theta$ (see Proposition~\ref{prop:tech}). In Proposition~\ref{prop:spectral_gap} and Theorem~\ref{thm:spectral}, we show how classical spectral gap estimates for reversible kernels (e.g. \citep{Dia91}) can be adapted to the nonreversible matrix $P_\ve$ via additive symmetrization. These bounds are then used to analyze the induced inhomogeneous model through the study of the noise sequence $v_n - \pi^*$.

More precisely, let $P_n = P_{\ve_n}$ denote the transition kernels of the induced inhomogeneous evolution model, and let $\pi_n = \pi_{\ve_n}$ be the invariant distribution of $P_n$. Let $Q_n$ denote the {\em pseudo-inverse} of $I-P_n$, i.e.\ the matrix with zero row sums that satisfies the Poisson equation
\begin{equation}
Q_n(I - P_n) = (I - P_n)Q_n = I - \Pi_n, \label{eq:pseudo}
\end{equation}
where $\Pi_n$ is the matrix with constant rows equal to $\pi_n$. By a slight abuse of notation, we refer to $Q_n$ as the pseudo-inverse of $P_n$.

Our main strategy is to use a suitable decomposition to bound $v_n - \pi^*$, relying on estimates for $\infn{Q_n}$, which in turn are obtained using the spectral gap bounds established earlier. For this analysis, we borrow ideas from classical stochastic approximation algorithms (\citep{Ben99,KusYin03,benaim2005stochastic}), originally developed for certain classes of inhomogeneous Markov chains, via the study of an associated differential inclusion (\citep{BenRai10}). In fact, this approach makes it possible to prove that $v_n$ converges to $\pi^*$ under the assumptions of Theorem~\ref{thm:main}. Roughly speaking, under some technical conditions (see \citep[Hypothesis~3.1]{BenRai10}) on the sequence $(P_n, Q_n, \pi_n)_n$, one can show that the continuous linear interpolation of $(v_n)_n$ is an {\em asymptotic pseudo-trajectory} (APT) of the ODE $\dot v = \pi^* - v$. It then follows that the limit set of $(v_n)_n$ is almost surely contained in the global attractor $\{\pi^*\}$. The APT approach has been successfully applied to other classes of inhomogeneous Markov chains, such as in \citep{BourCloez18} for finite state spaces (cf. Example~\ref{ex:cloez}) and \citep{benaim2017ergodicity} for a class of chains in $\RR^d$.

In our setting, the specific structure of the process allows us to show that the sequence $(P_n, Q_n, \pi_n)_n$ satisfies stronger conditions than those required in the general framework (see Proposition~\ref{prop:Q2}). This property enables us to establish almost sure convergence, together with explicit estimates on the rate at which $\EE(\infn{v_n - \pi^*})$ vanishes.

\subsection{Explicit formula}

Recall that $\co_0$ stands for the minimum cost among every $x$-tree in the directed graph induced by the cost function $\co$. We denote by $\mS_0$ the set of states $x$ in which $V(x)=0$, i.e.
\[
 \mS_0= \{x \in \mS\,:\, \tV(x)= \co_0\},
\]
and, for all $x\in \mS$, by  $\mG_x^0$  the set of $x$-trees where the value $\tV(x)$ is attained:
\[
\mG_x^0= \{g_x \in \mG_x \,:\, \co(g_x)=\tV(x)\}.
\]

We call $\mS_1$ and $\mG_x^1$ the complement of $\mS_0$ and $\mG_x^0$, respectively. Observe that with these definitions, $\cup_{x \in \mS_0}\mG_x^0= \mG^0$, where $\mG^0$ is the set of optimal trees defined in \eqref{eq:g0}. The following proposition, whose proof is deferred to Appendix~\ref{app:prop:tech}, plays a central role in our analysis as it provides an explicit formula for $\pi^*$ and establishes a precise rate of convergence of $\pi_\ve$ to $\pi^*$ in terms of the tree-optimality gap.

\vspace{2ex}
\begin{proposition}\label{prop:tech}
Let $(P_{\ve})_{\ve}$ be a homogeneous model of evolution with cost function $\co$. Let  $\pi_\ve$ be the invariant distribution of $P_\ve$. Then, the limit distribution $\pi^*$ is given by
\begin{equation} \label{eq:pi_explicit}
\pi^*(x)= \begin{cases}\dfrac{\sum_{g_x \in \mG_x^0} \prod_{e \in g_x} k(e)}{\sum_{y \in \mS_0}\sum_{g_y \in \mG_y^0} \prod_{e \in g_y} k(e)}, & x \in \mS_0,\\[1ex]
\phantom{0000000}0 &\text{otherwise}.
\end{cases}.
\end{equation}
Also, there exists a constant $C \geq 0$ such that $\infn{\pi_\ve-\pi^*}\leq C \ve^{\min\{\theta,1\}}$, where $\theta$ is the tree-optimality gap defined in \eqref{eq:theta}. 
In particular, if $\co$ takes integer values, then $\infn{\pi_\ve-\pi^*}\leq C \ve$.
\end{proposition}

\vspace{2ex}

\begin{remark} 
Since a model $(P_\ve)_\ve$ can be viewed as a particular family of Markov chains with rare transitions (cf.\ Remark~\ref{rem:MCRT}), we can directly apply \citep[Proposition 4.1]{Cat99} to obtain
\[
\lim_{\ve \rightarrow 0} \frac{\log \pi_{\ve}(x)}{\log(\ve)} = V(x).
\]
In Appendix~\ref{app:prop:tech}, we show precisely how this limit appears in our context (see \eqref{eq:piexp} for explicit computations).
\end{remark}

\vspace{0.5ex}
\subsection{Spectral gap estimates} Let $P$ be an irreducible Markov matrix on $\mS$, with invariant distribution $\pi$. We call $L^2(\pi)$ the euclidean space of functions $f:\mS \rightarrow \RR$, equipped with the scalar product
\[
\left<f \mid g \right>_{\pi} 
=\sum_{x \in \mS} f(x) g(x) \pi(x)\quad f,g: \mS \rightarrow \RR.
\]
Let $\pi(f)=\sum_{x \in \mS } f(x)\pi(x)$ be the expected value of $f$ and let $P^*$ be the adjoint of $P$, seen as a linear operator on $L^2(\pi)$. The adjoint operator is then given explicitly by
\[
P^*(x,y) = \frac{\pi(y)P(y,x)}{\pi(x)}, \quad x,y \in \mS. 
\]

 We are interested in lower bounds of $\lambda(P)$, the spectral gap of $P$, which is defined as
\[\lambda(P) := \inf \left\{\frac{\mE_P(f,f)}{\Var_{\pi}(f)} , \; \; f \in L^2(\pi), \; \, \Var_\pi(f) \neq 0\right\},\]
where
\begin{itemize}
\item[$\bullet$] $\Var_{\pi}(f)$ is the variance of $f$,
\[ \Var_{\pi}(f) := \|f - \pi(f)\|^2_{\pi} = \frac{1}{2} \sum_{x \neq y} \pi(x) \pi(y) (f(y) - f(x))^2,\] and\\
\item $\mE_P(f,f)$ is the Dirichlet form of $P$,
\[
\begin{aligned}
\mE_P(f,f) := \langle f \mid (I-P)f  \rangle_{\pi} &=& \sum_{x  \neq y} \pi(x)  f(x) P(x,y) \left( f(x) - f(y) \right) \\
&=&  \frac{1}{2} \sum_{x \neq y} \pi(x) P(x,y) (f(y) -f(x))^2.
\end{aligned}
\]
\end{itemize}

\vspace{2ex}
\begin{remark}
In the particular case where $P$ is reversible with respect to $\pi$, one has that $P^* = P$, and the spectral gap is equal to the quantity $1 - \lambda_2$, where $1 > \lambda_2 \geq ... \geq \lambda_{|\mS|}$ are the eigenvalues of the operator $P: L^2(\pi) \rightarrow L^2(\pi)$. Observe that in this case, $P$ is a symmetric contraction, that is, $||Pf||_{\pi} \leq ||{f}||_{\pi}$ for all $f \in L^2(\pi)$. For connections with alternative notions of symmetrization for nonreversible chains, see, for instance, \citep{chatterjee2025spectral}.
\end{remark}

 Let us call $\BFGamma$ the set of routing functions on $\mS$ and we denote by $\BFgamma \in \BFGamma$ a generic element, meaning that 
$$
\BFgamma:\mS \times \mS \to \bigcup\limits_{x,y \in \mS} \Gamma_{x,y}\quad  \text{and}\quad \BFgamma(x,y) \in \Gamma_{x,y},$$
that is, $\BFgamma(x,y)$ is a path joining $x$ and $y$. We call $e=(x_-,x_+) \in \mS\times \mS$ an admissible edge if $x_{-} \neq x_+$ and $P(x_-,x_+) >0$, and we denote by $E_P$ the set of oriented admissible edges associated to the kernel $P$.

The following is a well-known result from \citep{Dia91} adapted slightly to the case where irreducible Markov chain $P$ is not reversible with respect to $\pi$. 
\begin{proposition}\label{prop:spectral_gap} Let $P$ be an irreducible Markov matrix on $\mS$, with invariant probability measure $\pi$, and $M := \frac{P+P^*}{2}$. Then
\[\lambda(P) = \lambda \left( M\right)\geq \frac{1}{\inf_{\BFgamma \in \BFGamma}\kappa(\BFgamma)},\]
where
\[\kappa(\BFgamma) := \sup_{e=(x_-,x_+) \in E_M}\left\{ \frac{1}{\pi(x_-)M(x_-,x_+)} \sum_{(x,y): \, e \in \BFgamma(x,y)} |\BFgamma(x,y)| \pi(x) \pi(y)  \right\}, \]
and $|\BFgamma(x,y)|$ is the number of edges of the path $\BFgamma(x,y) \in \Gamma_{x,y}$.
\end{proposition}
\proof{Proof}
The proof of the equality follows from the fact that $\langle Pf\mid f \rangle_\pi=\langle f \mid P^*f \rangle_\pi ,$ implies that $\langle f \mid (I-P)f  \rangle_{\pi} = \langle f \mid (I-M)f  \rangle_{\pi}$. The lower bound is a direct application of \citep[Proposition 1\textquotesingle]{Dia91} to the kernel $M$, which is  reversible with respect to $\pi$. $\qed$

 The following result, whose proof is given in Appendix~\ref{app:thm:spectral} and follows the main ideas in \citep[Theorem 6.3]{Cat99} for the reversible case, provides an explicit lower bound for the spectral gap of $P_\ve$ as a function of $\ve$ and the energy barrier $\ec$.

\begin{theorem} \label{thm:spectral}
Let $(P_{\ve})_{\ve}$ be a homogeneous model of evolution with admissible cost function $\co$. Then, there exists some strictly positive constant $C$ such that 
\[
\lambda(P_{\ve}) \geq C \ve^{\ec}.
\]
\end{theorem}

\subsection{Estimations for the induced inhomogeneous chain}
The proof of the following result is left to Appendix~\ref{app:prop:Q2}.
\begin{proposition}
\label{prop:Q2} Let $(P_{\ve})_{\ve}$ be a homogeneous model of evolution, and $(\ve_n)_n$ be an $A$-vanishing mutation rate. Define $P_n= P_{\ve_n}$ and denote by $Q_n$ its pseudo-inverse matrix and by $\pi_n$ its unique invariant distribution. Then, there exist constants $C_i \geq 0$, $i=1,\ldots,4$, such that for all $n \geq 1$
\begin{itemize}\setlength\itemsep{0.7em}
\item[$(i)$] $ \infn{Q_n} \leq C_1 n^{A \ec}\log(\log(n+2))\log(n+1).$ 
If the potential function is such that $V\equiv 0$, then the logarithmic terms can be avoided.

\item[$(ii)$] $ \infn{P_{n+1} - P_n}\leq  C_2 \dfrac{1}{n^{1+\alpha}},$
where $0<\alpha \leq A$ is defined by
\begin{equation}\label{eq:alpha}
\alpha=
\begin{cases}
A \underline{c} &\text{if}\,\,\,\,\, 0<\underline{c}<1,\\
A &\text{otherwise},
\end{cases}
\end{equation}
and $\underline{c}$ is the minimum value of $\co$ over the set of pairs $(x,y)$ with strictly positive finite cost. If $\co$ takes integer values, then $\alpha=A$. 

\item[$(iii)$] $\infn{\pi_{n} - \pi^*}\leq C_3\dfrac{1}{n^{A\min\{\theta,1\}}},$ 
where $\theta$ is the tree-optimality gap defined in \eqref{eq:theta}. 

\item[$(iv)$] $\infn{\pi_{n+1} - \pi_{n}}\leq C_4 \dfrac{\log(\log(n+2))\log(n+1)}{n^{1+\alpha- A\ec}}.$
\end{itemize}
\end{proposition}

\subsection{Proof of Theorem~\ref{thm:main}: Conclusion} \label{sec:proof_conclusion} 

We have now all the ingredients to complete the proof of our main result. Let $(P_{\ve})_{\ve}$ be a model of evolution with ${\bf c}$ as its admissible cost function.  Let $(\ve_n)_n$ be a mutation rate satisfying the hypothesis of Theorem~\ref{thm:main}. Let us denote by $P_n=P_{\ve_n}$ and $\pi_n=\pi_{\ve_n}$ its invariant distribution. Then, for an initial condition $X_1=x_1$ and $n\geq 2$, 
\begin{equation}\label{eq:aux3}
\begin{aligned}
 v_n-\pi^*&=\frac{1}{n}\sum_{i=1}^n [\delta_{X_i} -\pi^*]=\frac{1}{n}(\delta_{x_1} - \pi^*)+ \frac{1}{n}\sum_{i=2}^n \left [\pi_{i-1} -\pi^* +\delta_{X_{i}}(I - \Pi_{i-1}) \right  ]\\
 &=\frac{1}{n}(\delta_{x_1} - \pi^*)+ \frac{1}{n}\sum_{i=2}^n \left [\pi_{i-1} -\pi^* +\delta_{X_{i}}(Q_{i-1} - P_{i-1}Q_{i-1}) \right  ],
 \end{aligned}
\end{equation}
where we have used the Poisson equation \ref{eq:pseudo}. Note that here we see the evolution on the space of signed measures, so matrices act from the right (as row vectors with the $\ell_1$ norm).

For all $i\geq 2$, set $\epsilon_i^0=\pi_{i-1}-\pi^*$ and{\setlength{\abovedisplayskip}{4pt}%
 \setlength{\belowdisplayskip}{4pt}%
\[
\delta_{X_{i}}(Q_{i-1} - Q_{i-1}P_{i-1})= \epsilon_i^1 + \epsilon_i^2+\epsilon_i^3,
\]}
where 
\[
\begin{aligned}
\epsilon_i^1 &= \delta_{X_{i}}Q_{i-1} - \delta_{X_{i-1}}Q_{i-1}P_{i-1},\\
\epsilon_i^2&=\delta_{X_{i-1}}Q_{i-1}P_{i-1}  -  \delta_{X_{i}}Q_{i}P_{i},\\
\epsilon_i^3&= \delta_{X_{i}} \left (Q_{i}P_{i} -  Q_{i-1}P_{i-1}\right ).
\end{aligned} 
\]

In order to keep the argument as simple as possible, we let $K \geq 0$ be a constant that may change from line to line, derived from the constants $C_1,C_2,C_3$ and $C_4$ obtained in Proposition~\ref{prop:Q2}. Also, $K$ may depend on the change from the $\ell_2$-norm to the infinity norm in $\RR^\mS$.
 
First, from Proposition~\ref{prop:Q2}$(iii)$, $U_n^0:=\frac{1}{n}\sum_{i=2}^n \epsilon_i^0$ goes to zero at the rate
\begin{equation*}
\infn{U_n^0}\leq \frac{K}{n}\sum_{i=2}^n \frac{1}{(i-1)^{A\min\{\theta,1\}}}=  \frac{K}{n}\sum_{i=1}^{n-1}\frac{1}{i^{A\min\{\theta,1\}}}.
\end{equation*}

In what remains of the proof, we denote $\varphi(n):=16 \log(\log(n+2))\log(n+1)$, for $n\in \NN$ so that $1\leq \varphi(n)\leq \varphi^2(n)$. Let $(\mF_n)_{n\geq 1}$ be the natural filtration of the process. Therefore, we have that for all $i \geq 2$, $\EE(\epsilon_i^1|\mF_{i-1})=0$ and $\EE(\infn{\epsilon_i^1}^2)\leq 4\infn{Q_{i-1}}^2$. From conditional independence and Proposition~\ref{prop:Q2}(i), 
\begin{equation} \label{eq:bound_un2}
 \sum_{i=2}^n \frac{\EE\infn{\epsilon_i^1}^2}{i^2} \leq 4  \sum_{i=2}^n \frac{\infn{Q_{i-1}}^2}{i^2}\leq K  \sum_{i=2}^n  \frac{\varphi^2(i)}{i^{2-2A\ec}}< +\infty.
 \end{equation}
So, the law of large numbers for martingale differences (see e.g. \citep{chow1967strong}) implies that $U_n^1:=\frac{1}{n} \sum_{i=2}^n \epsilon_i^1$ goes to zero almost surely. 
Additionally,
\begin{equation*}
\EE\infn{U_n^1}^2 \leq \frac{4}{n^2}\sum_{i=2}^{n} \infn{Q_{i-1}}^2\leq K\frac{\varphi^2(n)}{n^2}\sum_{i=1}^{n-1}i^{2A\ec}\leq K\frac{\varphi^2(n)}{n^{1-2A\ec}},
\end{equation*}
and, from Jensen's inequality
\[
\EE\infn{U_n^1}\leq \frac{K}{n^{1/2-A\ec}}\varphi(n).
\]
On the other hand, for all $n \geq 2$, a telescoping argument yields
\[
U_n^2:=\frac{1}{n} \sum_{i=2}^n \epsilon_i^2 = \frac{1}{n}\delta_{x_1}Q_{1}P_{1} -\frac{1}{n}\delta_{X_{n}}Q_{n}P_{n}.
\]
Consequently,
$$
\begin{aligned}
\infn{U_n^2}\leq \frac{K}{n}\left(1+n^{A\ec}\varphi(n) \right )\leq \frac{K}{n^{1-A\ec}}\varphi(n),
\end{aligned}$$
and we conclude that $U_n^2 \to 0$ almost surely.  Now, using that $P_nQ_n= Q_n+\Pi_n-I$, we define
\[
U_n^3:=\frac{1}{n} \sum_{i=2}^n \epsilon_i^3 =\frac{1}{n} \sum_{i=2}^n \left [\delta_{X_{i}}(Q_i-Q_{i-i}) +\delta_{X_{i}}(\Pi_{i}-\Pi_{i-1})\right ].
\]
Invoking the following identity \citep[Proposition 3.3]{BenRai10}
\[
Q_{n}-Q_{n-1} =Q_{n}(P_{n}-P_{n-1})Q_{n-1} - (\Pi_{n}-\Pi_{n-1})Q_{n-1},
\]
we obtain from Proposition~\ref{prop:Q2} that
\begin{align}
\infn{U_n^3}&\leq \frac{1}{n}\sum_{i=2}^{n} \Big [ \infn{Q_{i}}\infn{Q_{i-1}} \infn{P_{i} -P_{i-1}} + \infn{Q_{i-1}}\infn{\pi_{i} -\pi_{i-1}} + \infn{\pi_{i} -\pi_{i-1}} \Big ]\nonumber\\ 
&\leq \frac{K}{n}\varphi^2(n)\sum_{i=1}^{n-1} \frac{i^{2A\ec}}{i^{1+\alpha}} +  \frac{K}{n}\varphi(n) \sum_{i=1}^{n-1} \frac{i^{A\ec}}{i^{1+\alpha -A\ec}}\nonumber \\
&\leq \frac{K}{n}\varphi^2(n)\sum_{i=1}^{n-1} \frac{1}{i^{1+\alpha-2A\ec}}. 
\label{eq:bound_un3}
\end{align}
Collecting the bounds, from \eqref{eq:aux3}, and the fact that $1/2-A\ec < 1-A\ec< 1$, we get that
\[
\begin{aligned}
\EE \left (\infn{v_n -\pi^*} \right )&\leq \frac{1}{n}+ \sum_{i=0}^3 \EE \left (\infn{U_n^i}\right )\\
&\leq K\varphi^2(n)\left ( \frac{1}{n^{1/2-A\ec}} +\frac{1}{n}\sum_{i=1}^{n-1}\left [ \frac{1}{i^{A\min\{\theta,1\}}}+\frac{1}{i^{1+\alpha-2A\ec}}\right ]\right )\\
&\leq K\varphi^2(n)\left ( \frac{1}{n^{1/2-A\ec}} +\frac{1}{n}\sum_{i=1}^{n-1}\frac{1}{i^{A\min\{\theta,1\}}}\right ),
\end{aligned}
\]
where in the last line we used that $1+\alpha-2A\ec>1/2-A\ec$. Finally, if $A\min\{\theta,1\}>1$, then $\sum_{i=1}^{n-1}i^{-A\min\{\theta,1\}}$ is bounded and we conclude that $\EE(\infn{v_n -\pi^*})=\widetilde O(n^{-(1/2 - A\ec)} + n^{-1})$. On the other hand, if $A\min\{\theta,1\}\leq 1$, one has that $\frac{1}{n}\sum_{i=1}^{n-1}i^{-A\min\{\theta,1\}}=  \widetilde O(n^{-A\min\{\theta,1\}})$ and the proof is finished by noting that $1/2 - A\ec=A\min\{\theta,1\}$ if $A=\Acrit$. $\qed$

\vspace{2ex}
\noindent \textbf{Acknowledgments.}  The authors gratefully acknowledge the hospitality of the Institut de Mathématiques of the Université de Neuchâtel during the development of this project. The work of M. Bravo was partially supported by FONDECYT Grant No.1241805.  M. Faure acknowledges financial support from the French government under the “France 2030” investment plan managed by the French National Research Agency Grant ANR-17-EURE-0020, and by the Excellence Initiative of Aix-Marseille University - A*MIDEX.

\begin{appendix}
\section{Proof of Formula (\ref{eq:energy_ex1})} \label{app:energy_example}

\noindent \textsc{Case} $a\geq b$. Given that for all $x$, there is only one element in $\mG_x$, we have that for all $x \in \mS$
\[
\tV(x)= (N-|x|)b + Nb + a|x|=2Nb + (a-b)|x|,
\]
and therefore $ V(x)=(a-b)|x|$. On the other hand, for all $x=-N,\ldots, N-1$
\[
W(x,x+1)= W(x,x+1)=\min \{ (a-b)|x| + b,(a-b)(|x|-1) + a  \}= (a-b)|x|+b,
\]
then $\Elev(x,y)=  \max \{(a-b)|x|+b,(a-b)|y|+b \}$ and
\[\ec= \max_{x,y \in \mS}\,\left \{  \max \{(a-b)|x|+b,(a-b)|y|+b \} - (a-b)|x| - (a-b)|y| \right \}.\]
The previous expression attains its maximum at pairs $(x,0)$ and $(0,y)$ and therefore $\ec=b$.

\noindent \textsc{Case} $a < b$. Analogously to the previous case and given that the expression for $\tV$ is the same, we obtain $V(x)= (b-a)(N-|x|)$. Again, for all $x=-N,\ldots, N-1$
\[
W(x,x+1)=\min \{ (b-a)(N-|x|) + a,(b-a)(N-|x|+1) + b  \}= (b-a)(N-|x|) + a.
\]
For $x<y$, let $a(x,y)$ the value in $\mS$ where $\Elev(x,y)$ is attained, that is, the point $z \in \mS$ such that $|z|$ is minimal and that the edge $(z,z+1)$ belongs to the (unique) path linking $x$ and $y$.  Hence
\[
\begin{aligned}
\ec&= \max_{x,y \in \mS}\,\left \{  (b-a)(N-|a(x,y)|) + a - (b-a)(N-|x|) - (b-a)(N-|y|) \right \},\\
&= (b-a)\max_{x,y \in \mS}\,\left \{ |x|+ |y| - |a(x,y)|\right \}-(b-a)N + a.
\end{aligned}
\]
This expression is maximized at pairs $(-N,N)$ and $(N,-N)$, since in that case $a(x,y)=0$, and therefore $\ec=(b-a)N+a$.

\section{Proof of Theorem \ref{thm:comparison}}\label{app:thm:comparison}
 Recall that $\Gamma_{z',z}$ is the set of paths (with possible infinite cost) from $z'$ to $z$.   Let $x,y \in \mS$, $x\neq y$ and $z \in \mC$, the union of the recurrent classes induced by $\Plim$.  It is sufficient to prove that there exists $\hat{\gamma} \in \Gamma_{x,y}$ such that
\[\max_{e \in \hat{\gamma}} W(e) -V(x) - V(y) \leq \CR(z)=\max_{z' \neq z } \min_{\gamma \in \Gamma_{z',z}} \co(\gamma),\]
and we assume without loss of generality that $x \neq z$. By definition of $\CR(z)$ there exists $\gamma_x =(x_0,x_1,...,x_m) \in \Gamma_{x,z}$ and $\gamma_y = (y_0,y_1,...,y_\ell) \in \Gamma_{y,z}$ such that $\co(\gamma_x) \leq \CR(z), \; \; \co(\gamma_y) \leq \CR(z)$,
where $\gamma_y$ is possibly empty if $y=z$ with $\co(\gamma_y)=0$ in that case. Let 
$$\hat{\gamma} := (x_0,...,x_{m-1},z,y_{\ell-1},...,y_0) \in \Gamma_{x,y}$$ and $g_x$ be an $x$-tree such that 
$$\min_{g'\in \mG_x}\co(g')=\co(g_x)=\tV(x).$$ 

For all $i \in \{1,...,m-1\}$, the idea is to construct an $x_i$-tree, $g_{x_i} \in \mG_{x_i}$, by adding the subpath $(x_0,x_1)(x_1,x_2)\ldots(x_{i-1},x_i)$ of $\gamma_x$ to $g_x$, while removing some edges to ensure that the result is in $\mG_{x_i}$. Explicitly, we modify $g_x$ as follows:

\noindent $\bullet$ For $j=1,...,i$, add the edge $(x_{j-1},x_{j})$  and remove the (unique) edge starting from $x_j$ in $g_x$.

We then have
\[\co(g_{x_i}) \leq \tV(x) + \sum_{j=1}^{i} \co(x_{j-1},x_{j}).\]
Using that $\min_{g'\in \mG_{x_i}}\co(g')=\tV(x_i)\leq \co(g_{x_i})$ and adding $\co(x_i,x_{i+1})$ on both sides, we get
\[\tV(x_i) + \co(x_i,x_{i+1}) \leq \tV(x) + \sum_{j=0}^{i+1} \co(x_j,x_{j+1}) \leq \tV(x) + \co(\gamma_x) \leq  \tV(x) + \CR(z).\]
Using that $\tV(x_i) -\tV(x) = V(x_i) - V(x)$, and $W(x_i,x_{i+1}) \leq V(x_i) + \co(x_i,x_{i+1})$  we obtain
\begin{equation*} 
W(x_i,x_{i+1}) \leq V(x_i) + \co(x_i,x_{i+1}) \leq V(x) + \CR(z), \quad i=0,\ldots, m-1,
\end{equation*}
where for $i=0$ the inequality holds since $x_0=x \neq z$. If $y=z$, the result is already proven. 

Assuming that $\ell \geq 1$, the same argument applied now to the path $\gamma_y$ gives, for $j=0,...,\ell-1$,
\begin{equation*}
 W(y_{j+1},y_j)=W(y_{j},y_{j+1}) \leq V(y_{j}) + \co(y_{j},y_{j+1}) \leq V(y) + \co(\gamma_y) \leq V(y) +\CR(z),
\end{equation*}
from the symmetry of the potential $W$. Finally, using the two inequalities above, we get
\[\max_{e \in \hat \gamma} W(e) \leq \max \{V(x),V(y)\} + \CR(z),\]
and the conclusion follows. $\qed$

\section{Proof of Proposition \ref{prop:tech}}\label{app:prop:tech}
 We invoke the following well-known explicit formula for $\pi_\ve$  (see e.g.  \citep{Cat99}, Lemma 3.2),
\begin{equation} \label{eq:pi_catoni}
\pi_{\ve}(x) = \dfrac{\sum_{g_x \in \mG_x} \prod_{e \in g_x} P_{\ve}(e)}{\sum_{y \in \mS} \sum_{g_y \in \mG_y} \prod_{e \in g_y} P_{\ve}(e)},\quad x \in \mS,
\end{equation}
where we want to estimate the behavior of $\pi_{\ve}(x)$ as $\ve$ is close to zero for $x \in \mS$. Let us define, for all $x \in \mS$ and $g_x \in \mG_x$ 
\[
P_{\ve}(g_x):=\prod_{e \in g_x}P_{\ve}(e)= \ve^{\co(g_x)} K^{g_x}(\ve), \,\,\,\text{where} \quad K^{g}(\ve):= \prod_{e \in g} k(e)(1+h_e(\ve)).
\]
In what follows, we assume without loss of generality that $h_e:[0,1]\to [-1/2,1/2]$ and we let $L$ be a uniform bound for the corresponding Lipchitz constants. Observe that, by construction, the functions $ K^{g}$ are lower bounded by some strictly positive quantity since $K^{g}(0)= \prod_{e \in g} k(e)>0$. 

Recalling that $\co_0$ is the minimal cost among all trees, we have that, for all $x \in \mS$,
\[
\sum_{g_x \in \mG_x}P_{\ve}(g_x)= \ve^{\tV(x)}\Big (\sum_{g_x \in \mG_x^0} K^{g_x}(\ve) + \sum_{g_x \in \mG_x^1}\ve^{\co(g_x)-\tV(x)}K^{g_x}(\ve)\Big ),
\]
and
\[
\begin{aligned}
\sum_{y \in \mS}\sum_{g_y \in \mG_y}P_{\ve}(g_y)= \ve^{\co_0}\sum_{y \in \mS_0}\Big (\sum_{g_y \in \mG_y^0} K^{g_y}(\ve) &+ \sum_{g_y \in \mG_y^1}\ve^{\co(g_y)-\co_0}K^{g_y}(\ve)\Big ) +\\
&+\ve^{\co_0}\sum_{y \in \mS_1}\sum_{g_y \in \mG_y}\ve^{\co(g_y)-\co_0}K^{g_y}(\ve).
\end{aligned}
\]
Therefore, using that $V(x)=\tV(x)-\co_0$, the expression for $\pi_\ve$ in \eqref{eq:pi_catoni} becomes
\begin{equation} \label{eq:piexp} 
    \pi_\ve(x)=\ve^{V(x)}\dfrac{\sum_{ \mG_x^0} K^{g_x}(\ve) + \sum_{\mG_x^1}\ve^{\co(g_x)-\tV(x)}K^{g_x}(\ve)}{\sum_{ \mS_0}\sum_{\mG_y^0} K^{g_y}(\ve) + \sum_{\mS_0}\sum_{\mG_y^1}\ve^{\co(g_y)-\co_0}K^{g_y}(\ve)+\sum_{\mS_1}\sum_{\mG_y}\ve^{\co(g_y)-\co_0}K^{g_y}(\ve)},
\end{equation}
for all $x \in \mS$. Noting that all the exponents of $\ve$ in \eqref{eq:piexp} are strictly positive and taking $\ve \to 0$ lead to formula \eqref{eq:pi_explicit}.

For the remaining part of the proof, we denote by $\pi_\ve^*$ the distribution obtained by evaluating \eqref{eq:pi_explicit} at some $\ve >0$, that is, 
\[
\pi_\ve^*(x)= \begin{cases}\dfrac{\sum_{g_x \in \mG_x^0} K^{g_x}(\ve)}{\sum_{y \in \mS_0}\sum_{g_y \in \mG_y^0} K^{g_y}(\ve)}, & x \in \mS_0,\\[1ex]
\phantom{0000000}0 &\text{otherwise}.
\end{cases}
\]
Lemma~\ref{lem:pi_eps} below shows that there exists $C \geq 0$ such that $\infn{\pi_\ve^*-\pi_\ve}\leq C \ve^{\theta}$. Now, using that $ K^{g}(\ve)= \prod_{e \in g} k(e)(1+h_e(\ve))$, a routine computation shows that $|\pi^{*}_{\ve}(x)'|\leq 2L\overline C \ve$, 
for $x \in \mS_0$, where we assume for simplicity  that $h_e$ is differentiable. The conclusion follows from the fact that $\infn{\pi_\ve -\pi^*}\leq \infn{\pi_\ve -\pi^*_\ve}+ \infn{\pi_\ve^* -\pi^*},$ and taking $C= 2\overline C(L+1)$. $\qed$

\begin{lemma}\label{lem:pi_eps} There exists $C \geq 0$ such that $\infn{\pi_\ve^*-\pi_\ve}\leq C \ve^{\theta}$.
\end{lemma}
\proof{Proof}Let us start assuming that $x \in \mS_0$, that is that $\mG_x$ contains an $x$-tree of optimal cost among all trees. In that case $V(x)=0$ and $\tV(x)=\co_0$. A simple algebraic manipulation allows us to write
\[
\pi_\ve^*(x)-\pi_\ve(x)= \varphi^1_\ve(x)+\varphi^2_\ve(x),
\]
where
\begin{align*}
\varphi^1_\ve(x)&=\pi^*_\ve(x) \dfrac{\sum_{\mS_0}\sum_{\mG_y^1}\ve^{\co(g_y)-\co_0}K^{g_y}(\ve)+\sum_{\mS_1}\sum_{\mG_y}\ve^{\co(g_y)-\co_0}K^{g_y}(\ve)}{\sum_{ \mS_0}\sum_{\mG_y^0} K^{g_y}(\ve) +\sum_{\mS_0}\sum_{\mG_y^1}\ve^{\co(g_y)-\co_0}K^{g_y}(\ve)+\sum_{\mS_1}\sum_{\mG_y}\ve^{\co(g_y)-\co_0}K^{g_y}(\ve)},\\ 
\varphi^2_\ve(x)&=-\dfrac{\sum_{\mG_x^1}\ve^{\co(g_x)-\co_0}K^{g_x}(\ve)}{\sum_{ \mS_0}\sum_{\mG_y^0} K^{g_y}(\ve) + \sum_{\mS_0}\sum_{\mG_y^1}\ve^{\co(g_y)-\co_0}K^{g_y}(\ve)+\sum_{\mS_1}\sum_{\mG_y}\ve^{\co(g_y)-\co_0}K^{g_y}(\ve)}. 
\end{align*}
Recall that sums over $ \mG_y^1$ are zero by definition if $\mG_y^1=\emptyset$.
Then, 
\[
\begin{aligned}
|\varphi^1_\ve(x)|&=\pi^*_\ve(x)  \frac{\sum_{\mS_0}\sum_{\mG_y^1}\ve^{\co(g_y)-\co_0}K^{g_y}(\ve)+\sum_{\mS_1}\sum_{\mG_y}\ve^{\co(g_y)-\co_0}K^{g_y}(\ve)}{\sum_{ \mS_0}\sum_{\mG_y^0} K^{g_y}(\ve) +\sum_{\mS_0}\sum_{\mG_y^1}\ve^{\co(g_y)-\co_0}K^{g_y}(\ve)+\sum_{\mS_1}\sum_{\mG_y}\ve^{\co(g_y)-\co_0}K^{g_y}(\ve)}\\
&\leq \ve^{\theta} \frac{\sum_{\mS_0}\sum_{\mG_y^1}K^{g_y}(\ve)+\sum_{\mS_1}\sum_{\mG_y}K^{g_y}(\ve)}{\sum_{ \mS_0}\sum_{\mG_y^0} K^{g_y}(\ve)}\leq \overline C \ve^\theta,
\end{aligned}
\]
with 
\[
\overline C:= \max_{\ve \in [0,1]}\frac{\sum_{\mS_0}\sum_{\mG_y^1}K^{g_y}(\ve)+\sum_{\mS_1}\sum_{\mG_y}K^{g_y}(\ve)}{\sum_{ \mS_0}\sum_{\mG_y^0} K^{g_y}(\ve)},
\]
and $\theta$ is the tree-optimality gap defined in \eqref{eq:theta}, using that $\left(\cup_{y \in \mS_0} \mG_y^1 \right)\cup \left(\cup_{y \in \mS_1} \mG_y\right)=\mG^1$. If $\mG^1 =\emptyset$  then $\mS_0=\mS$ and $\tV$ is constant. Hence $\varphi^1_\ve(x)=\pi^*_\ve(x)$ and $\varphi^2_\ve(x)=0$ for all $x\in \mS$ and the formula is valid with $\theta =+\infty$.
Similarly, 
\[
\begin{aligned}
|\varphi_\ve^2(x)|\leq \ve^{\theta} \frac{\sum_{g_x \in \mG_x^1}K^{g_x}(\ve) }{\sum_{y \in \mS_0}\sum_{g_y \in \mG_y^0} K^{g_y}(\ve)}\leq \overline C\ve^{\theta}.
\end{aligned}
\]
We conclude that for all $x \in \mS_0$, $|\pi_\ve(x) -\pi^*_\ve(x)|\leq |\varphi^1_\ve(x)|+ |\varphi^2_\ve(x)|\leq 2 \overline C\ve^\theta$.

It remains to prove that the bound holds when $x \in \mS_1$. In such case, $V(x)>0$ and $\pi_\ve^*(x)=0$. Using that $\ve \leq 1$ and that $\mG_x= \mG_x^0 \cup \mG_x^1$, we have
\[
|\pi_\ve(x) -\pi^*_\ve(x)|\leq \ve^{V(x)} \frac{\sum_{g_x \in \mG_x} K^{g_x}(\ve)}{\sum_{y \in \mS_0}\sum_{g_y \in \mG_y^0} K^{g_y}(\ve)}\leq c\ve^{V(x)}.
\]
Notice that $V(x) \geq \theta$ since $V(x)=\co(g_x) -\co_0$ for some $g_x \in \mG_x^0$, which does not contains globally optimal trees by definition. In conclusion $\infn{\pi_\ve -\pi^*_\ve}\leq 2 |\mS| \overline C \ve^{\theta}$, as claimed (seen as the $\ell_1$ norm). $\qed$

\section{Proof of Theorem \ref{thm:spectral}}\label{app:thm:spectral} Recall that $\BFGamma$ stands for the set of routing functions and that $\pi_\ve$ is the invariant probability distribution of $P_\ve$. By Proposition~\ref{prop:spectral_gap}, we have, for any $\ve > 0$ and every $\BFgamma \in \BFGamma$,
\[\lambda(P_{\ve}) \geq \frac{1}{\kappa_{\ve}(\BFgamma)},\]
where
\[\kappa_{\ve}(\BFgamma) := \sup_{e=(x_{-},x_{+}) \in E_M}\left\{ \frac{1}{\pi_{\ve}(x_-)M_{\ve}(x_-,x_+)} \sum_{(x,y): \, e \in \BFgamma(x,y)}|\BFgamma(x,y)| \pi_{\ve}(x) \pi_{\ve}(y)  \right\},\]
and $M_\ve(x,y)=\frac{1}{2}(P_\ve(x,y) + P_\ve^*(x,y))$. Hence, we aim to find some routing function $\overline \BFgamma \in \BFGamma$ such that $\kappa_{\ve}(\overline \BFgamma) \leq C^{-1}\ve^{-\ec}$ for some constant $C>0$ and all $\ve \in (0,1)$. From \eqref{eq:piexp}, there exist constants
$\underline{K},\overline{K}>0$ such that
\[
\underline{K}\,\ve^{V(x)} \leq \pi_\ve(x) \leq \overline{K}\,\ve^{V(x)}, \qquad x\in\mS.
\]
Thus,  recalling that $|h_{x,y}(\ve)|\leq 1/2$, we obtain, for $x,y$ such that $\co(x,y)<\infty$,
\[
\frac{1}{2}\,\underline{K}\,\ve^{V(x)+\co(x,y)} \leq \pi_\ve(x)P_\ve(x,y),
\qquad\text{and}\qquad
\pi_\ve(x)\pi_\ve(y) \leq \frac{9}{4}\,\overline{K}\,\ve^{V(x)+V(y)},
\]
for all $x,y\in\mS$. This, combined with the irreducibility property, gets that for all admissible edge $e=(x_{-},x_+) \in E_M$,
\[
\begin{aligned}
\pi_{\ve}(x_{-}) M_{\ve}(x_{-},x_{+}) &= \frac{1}{2} \left(\pi_{\ve}(x_{-}) P_{\ve}(x_{-},x_{+}) + \pi_{\ve}(x_{+}) P_{\ve}(x_{+},x_{-}) \right)\\
&\geq \frac{1}{2}\max \{\pi_{\ve}(x_{-}) P_{\ve}(x_{-},x_{+}), \pi_{\ve}(x_{+}) P_{\ve}(x_{+},x_{-}) \}\\
&\geq \frac{1}{4}\underline{K}\max\{\ve^{V(x_{-})+\co(x_{-},x_{+})}, \ve^{V(x_+)+\co(x_{+},x_{-})} \}\\
&\geq \frac{1}{4}\underline{K}\ve^{\min \{V(x_{-})+\co(x_{-},x_{+}),V(x_+)+\co(x_{+},x_{-})\}}\\
&=\frac{1}{4}\underline{K} \ve^{W(x_{-},x_{+})}.
\end{aligned}
\]
Consequently, for any $e=(x_{-},x_+)\in E_M$, and any routing function $\BFgamma \in \BFGamma$,
\[
\frac{1}{\pi_{\ve}(x_-)M_{\ve}(x_-,x_+)} \sum_{\substack{x,y \in \mS:\\ e \in \BFgamma(x,y)}}|\BFgamma(x,y)| \pi_{\ve}(x) \pi_{\ve}(y)\leq \frac{9\overline K}{\underline K}\sum_{\substack{x,y \in \mS:\\ e \in \BFgamma(x,y)}} |\BFgamma(x,y)|\ve^{V(x) + 
 V(y)-W(e)}.
\] 
 Now let $\overline \BFgamma \in \BFGamma$ be such that  $\overline \BFgamma(x,y)$ is a minimizing path in the sense that
\[
\Elev(x,y) = \max_{e \in \overline \BFgamma(x,y)} W(e), \quad  \text{ for all} \,\, (x,y) \in \mS \times \mS.
\]
Recalling that $\ec=\max_{x,y} \Elev(x,y)-V(x) -V(y)$, we finally have that
\[\kappa_{\ve}(\overline \BFgamma) \leq  \frac{9\overline K}{\underline K} |\mS|^3\ve^{-\ec}:=C^{-1}\ve^{-\ec}.\qquad  
\]

\section{Proof of Proposition \ref{prop:Q2}} \label{app:prop:Q2}
In order to make computations easier, we assume for simplicity $\ve_n =  n^{-A}$ for all $n \geq 1$. For a mutation rate $(\ve_n)_n$ that is $A$-vanishing, all the calculations presented below will be essentially the same, except that the constants will be slightly different.

We know from Theorem~\ref{thm:spectral} that there exists $C \geq 0$ such that
\[
\lambda(P_n) \geq C \ve_n^{\ec}= \frac{C}{n^{A\ec}}.
\]
Now, estimates due to \citep[Chapter 2]{Sal97} (see also  \citep[Proposition 3.4]{BenRai10}) allow to bound the rate in which the infinity norm of the pseudo-inverse matrix $Q_n$ of $P_n$ increases. Precisely, we have that
\[
 \sum_{y \in \mS}|Q_n(x,y)| \leq  \dfrac{|\mS|}{\lambda(P_n)} \left \{\log_+ \left ( \log \left ( \frac{1}{\overline \pi_n}\right )\right )\left ( \frac{\log \left (\frac{1-\overline\pi_n}{\overline \pi_n} \right )}{1-2\overline\pi_n}\right ) +e \right  \}, \quad x\in \mS,
 \]
 where $\overline \pi_n= \min_{x \in \mS}\pi_n(x)$ and $\log_+ (t)= \max \{0, \log(t) \}$. Using the notation introduced in the proof of Theorem~\ref{thm:spectral}, and calling $\overline v=\max_{x \in \mS}V(x)$, we have that, for all $n \geq 1$, 
\[
\underline{K} \ve_n^{\overline v}= \frac{\underline{K}}{n^{A \overline v}}\leq \overline\pi_n\leq \frac{1}{2}.
\]
Assuming that $\overline v>0$, we use the trivial bound $\frac{\log \left (\frac{1-t}{t} \right )}{1-2t}\leq \frac{2\log \left (\frac{1}{t} \right )}{1-t}$ for $t\in (0,1/2]$, to get 
\[
\frac{\log \left (\frac{1-\overline\pi_n}{\overline\pi_n} \right )}{1-2\overline\pi_n}\leq \frac{2\log \left (\frac{1}{\overline\pi_n} \right )}{1-\overline\pi_n}\leq 4\log(n^{A \overline v}/\underline{K}).
\]
Hence 
\[
 \infn{Q_n} \leq
 \dfrac{|\mS|}{C} n ^{A \ec}(4\log_+(\log(n^{A \overline v}/\underline{K} ) \log(n^{A \overline v}/\underline{K}) +e),
\]
which we cast as 
\[
 \infn{Q_n} \leq
 C_1  n ^{A \ec}\log(\log(n+2))\log(n+1)
\]
for an appropriate constant $C_1$.  Observe that if $\overline v=0$, then the logarithmic dependence on $n$ can be avoided since $\overline\pi_n$ is lower bounded by $\underline{K}$. 

We now turn to point $(ii)$. Recall that we assume for simplicity that $h_e:[0,1]\to [-1/2,1/2]$ and we let $L$ be a uniform bound for the corresponding Lipchitz constants. 

Let $x,y$ be such that $\co(x,y) = 0$. In that case,
$P_{\ve}(x,y) = k(x,y) (1 + h_{x,y}(\ve))$, so that
\[\left|P_{n+1}(x,y) - P_n(x,y) \right|  =  |P_{\ve_{n+1}}(x,y) - P_{\ve_n}(x,y)| \leq  A L k(x,y) |\ve_{n+1} - \ve_n| \leq \frac{ Lk(x,y)}{n^{1+A}}.\]

If $\co(x,y) = c >0$, then $P_{\ve}(x,y) = k(x,y) \ve^c (1 + h_{x,y}(\ve))$. So, for $\ve \in (0,1)$, $|\phi'(\ve)|\leq 2\ve^{c-1} + L\ve ^c \leq (2+L)\ve^{c-1}$, with $\phi(\ve)= \ve^c (1+h_{x,y}(\ve))$. Hence 
\[\left|P_{n+1}(x,y) - P_n(x,y) \right|  \leq (2+L)k(x,y) \ve_n^{c-1}|\ve_{n+1} - \ve_n|\leq A\frac{(2+L)k(x,y)}{n^{1+A c}}.\]

 Combining the previous bounds, we get 
\[
\infn{P_{n+1}- P_n} \leq \infn{k}\frac{A(2+L)}{n^{1+\alpha}},
\]
where $\alpha$ is defined in Equation~\eqref{eq:alpha}. Part $(iii)$ follows directly from Proposition~\ref{prop:tech} and the precise form of the mutation rate. Finally, to prove $(iv)$, we observe that, by definition,
$$(\pi_{n+1}-\pi_n)(I-P_n)=\pi_{n+1}(P_{n+1}-P_n).$$

So, multiplying by $Q_n$ and using the definition of the pseudo-inverse matrix \eqref{eq:pseudo}, we get
$\pi_{n+1}-\pi_n= \pi_{n+1}(P_{n+1}-P_n)Q_n$, 
and then
\[
\infn{\pi_{n+1}-\pi_n}\leq \infn{P_{n+1}-P_n}\infn{Q_n}\leq C_1C_3 \frac{\log(\log(n+2))\log(n+1)}{n^{1+\alpha - A\ec}}. \qquad \qed
\]


\section{Proof of Proposition~\ref{prop:cloez}}\label{app:examples}
For the sake of clarity, we start with the simpler case given by Example~\ref{ex:Dietz} by slightly modifying the argument given in Section~\ref{sec:proof_conclusion}. Let us then adopt all the notation introduced there. Recall that $V\equiv 0$, every $x$-tree has cost $N-1$, and $\theta=+\infty$, so that \eqref{eq:piexp} becomes
\begin{equation*} 
    \pi_\ve(x)=\frac{\sum_{ g_x\in \mG_x} K^{g_x}(\ve) }{\sum_{y \in  \mS}\sum_{g_y\in \mG_y} K^{g_y}(\ve)},
\end{equation*}
where, by definition, $K^{g_x}(\ve)=\prod_{e \in g_x} k(e)(1+h_e(\ve))=\prod_{e \in g_x} k(e)$, since $h_e\equiv 0$. Consequently $ \pi_\ve=\pi^*$ for all $\ve \in (0,1]$ and therefore $U_n^0\equiv 0$.

Since $\ec=1$, the bound on $U_n^2$ becomes  $\infn{U_n^2}\leq K/n^{1-A}$. Moreover,
since $\co(x,y)=1$ for all $x \neq y$ and $\alpha=A$, the bound \eqref{eq:bound_un3} is now
\[
\infn{U_n^3}\leq \frac{K}{n}\sum_{i=1}^{n-1} \frac{1}{i^{1-A}}\leq \frac{K}{n^{1-A}}.
\]
Let us show that $U_n^1$ vanishes almost surely for $A<1$. In this model $P_\ve$ approaches the identity matrix as $\infn{I-P_\ve}=2\max_{x\in \mS}\{1-P_\ve(x,x)\}=2\ve\max_{x\in \mS}\sum_{y \neq x}k(x,y)\leq 2 \infn{k}\ve$. Thus
\[
\begin{aligned}
\EE(\infn{\epsilon_i^1}^2| X_{i-1} = x)&=\infn{\delta_{x}Q_{i-1}(I-P_{i-1})}^2P_{i-1}(x,x) + \sum_{y\neq x}\infn{\delta_yQ_{i-1}-\delta_{x}P_{i-1}Q_{i-1}}^2P_{i-1}(x,y)\\[-1ex]
&\leq \infn{Q_{i-1}}^2\infn{I-P_{i-1}}^2 + 4\infn{Q_{i-1}}^2\ve_{i-1}\\
&\leq  \infn{Q_{i-1}}^2(4 \infn{k}^2\ve_{i-1}^2+ 4\ve_{i-1})\\
&\leq 8 \infn{Q_{i-1}}^2 \ve_{i-1},
\end{aligned}
\]
where we have assumed without loss of generality that $\infn{k}^2 \ve_{i-1}\leq 1$. Taking expectation in the inequality above, we obtain
\begin{equation*}
 \sum_{i=2}^n \frac{\EE(\infn{\epsilon_i^1}^2)}{i^2} \leq 8  \sum_{i=2}^n \frac{\infn{Q_{i-1}}^2\ve_{i-1}}{i^2}\leq 8\sum_{i=2}^n  \frac{i^{2A-A}}{i^{2}}=8\sum_{i=2}^n  \frac{1}{i^{2-A}}< +\infty,
 \end{equation*}
and hence $U_n^1 \to 0$, almost surely. Also,
\begin{equation*} 
\EE(\infn{U_n^1}^2) \leq \frac{8}{n^2}\sum_{i=2}^{n} \infn{Q_{i-1}}^2\ve_{i-1}\leq \frac{8}{n^{1-A}},\quad \text{and thus}  \quad \EE\infn{U_n^1} \leq \frac{2 \sqrt{2}}{n^{(1-A)/2}}.
\end{equation*}

Putting everything together, we get that $v_n \to \pi^*$ almost surely. Finally, as claimed, the convergence rate in this case can be estimated now as
\[
\begin{aligned}
\EE \left (\infn{v_n -\pi^*} \right )&\leq \frac{1}{n}+ \sum_{i=0}^3 \EE \left (\infn{U_n^i}\right ) =O\left (  \frac{1}{n^{(1-A)/2}}\right ).\\
\end{aligned}
\]
We now turn to the more general case of Example~\ref{ex:cloez}. Let us write, for $x \neq y$,
\[
\widetilde P_n(x,y)=\begin{cases}
     \ve_n(k(x,y) +\err_n(x,y)),   & \text{if}\quad  k(x,y) >0,\\
      \ve_n \,\err_n(x,y)   & \text{if}\quad  k(x,y) =0.
 \end{cases}
\]

As we mentioned, this chain cannot be written directly as an inhomogeneous model induced by a homogeneous counterpart, because of the vanishing terms $\err_n(x,y)$. However, all our arguments can be adapted to this setting by means of classical perturbation arguments.

Indeed, consider the homogeneous model $( P_\ve)_\ve$ with cost given by $\co(x,x)=0$, $\co(x,y)=1$ if $k(x,y)>0$, and $+\infty$ otherwise. Notice that, in order to fit our definition, we may assign an arbitrary value $k(x,y)>0$ whenever the cost is not finite. Therefore, by irreducibility of $k$, we have $\tV(x)=N-1$, $V\equiv 0$, $\ec=1$, and $\theta=+\infty$.

Let us denote $P_{n}:= P_{\ve_n}$, with $Q_n$ its pseudo-inverse, $\pi_n \equiv \pi^*$ the invariant distribution, and we set $\err_n=\max_{x \neq y}|\err_n(x,y)|$. Using that $|P_n(x,y)-\widetilde P_{n}(x,y)|=O(\ve_n \err_n)$ for all $x,y \in \mS$, and that $\infn{Q_n}\leq K\ve_n^{-1}$, we obtain
$$\infn{\widetilde \pi_n-\pi^*}\leq \infn{\widetilde P_{n}-P_n}\infn{Q_n}\leq K \err_n,$$
where $\widetilde \pi_n$ is the invariant distribution of $\widetilde P_n$, and $K$ is a positive constant that may change from line to line to avoid notation clutter. Also, for all $x \neq y$,
\[
\scalebox{0.98}{$|\pi^*(x) P_n(x,y) \!-\!\widetilde \pi_n(x) \widetilde P_n(x,y)|\leq \pi^*(x) |P_n(x,y)\! -\! \widetilde P_n(x,y)|\!+\! |\pi^*(x)\!-\!\widetilde \pi_n(x)|\widetilde P_n(x,y)\leq K\ve_n \err_n$},
\]
where we used that $\widetilde P_n(x,y)\leq K\ve_n$ for $x \neq y$. Thus, for all $f:\mS \to \RR$, the respective Dirichlet forms can be compared as
\[
\scalebox{0.96}{$| \mathcal E_{P_n}(f,f)\! -\! \mathcal E_{\tilde P_n}(f,f)|\!=\! \frac{1}{2}\left |  \sum \limits_{x \neq y}\!( \pi^*(x) P_n(x,y) \!-\!\widetilde \pi_n(x) \widetilde P_n(x,y))(f(y)\! \!-\!\!f(x))^2 \right |\leq \!K \ve_n \err_n \Var_{\pi^*}(f)$},
\]
where a factor $1/\min_x \pi^*(x)$ squared is absorbed into the constant. 

Now, using that $|\Var_{\pi^*}(f) - \Var_{\widetilde \pi_n}(f)| \leq K \err_n$ and taking $f$ such that $\Var_{\widetilde \pi_n}(f)=1$, we obtain
$$\mathcal E_{\tilde P_n}(f,f) \geq \mathcal E_{P_n}(f,f) -K \ve_n \err_n \quad \text{so} \quad \lambda(\widetilde P_n)\geq \inf_{\Var_{\widetilde \pi_n}(f)=1}\mathcal E_{P_n}(\!f,\!f) -K \ve_n \err_n,$$
by taking the infimum. Setting $g = f / \left(\Var_{\pi^*}(f)\right)^{1/2}$, we have $\Var_{\pi^*}(g)=1$ (taking $\err_n$ small without loss of generality). Hence, $\mathcal E_{P_n}(f,f)=\Var_{\pi^*}(f)\mathcal E_{P_n}(g,g) \geq \Var_{\pi^*}(f)\lambda(P_n)$. Therefore, using that $\lambda(P_n)\geq K\ve_n$ and that $\Var_{\pi^*}(f)\geq 1-K\err_n$,
\[
\lambda(\widetilde P_n) \geq \lambda(P_n)\inf_{\Var_{\widetilde \pi_n}(f)=1}\Var_{\pi^*}(f) - K\ve_n \err_n \geq \lambda(P_n)(1-K\err_n) - K\ve_n\err_n \geq K \ve_n.
\]
Let us now see how the estimates given by Proposition~\ref{prop:Q2} translate to the sequence $(\widetilde P_n)_n$. From the previous analysis, we know that $\infn{\widetilde \pi_n - \pi^*} \leq K \err_n$ and that $\infn{\widetilde Q_n} \leq K n^{-A}$, where $\widetilde Q_n$ denotes the pseudo-inverse of $\widetilde P_n$. Moreover, by the triangle inequality, $\infn{\widetilde P_{n+1} - \widetilde P_n} \leq K n^{-(1+A)}$, and therefore $\infn{\widetilde \pi_{n+1}-\widetilde \pi_n}\leq \infn{\widetilde P_{n+1}-\widetilde P_n}\infn{\widetilde Q_n} \leq K n^{-1}$. Hence, the bounds for $U_n^i$, $i \in \{0,\ldots,3\}$, take the form
\begin{equation*} \infn{U_n^0}\leq \frac{K}{n}\sum_{i=1}^{n}\err_i, \quad \EE\infn{U_n^1} \leq \frac{2 \sqrt{2}}{n^{(1-A)/2}}, \quad \infn{U_n^2}\leq \frac{K}{n^{1-A}}, \quad \text{and}\quad \infn{U_n^3}\leq K \frac{\ln(n+1)}{n}. \end{equation*}

The almost sure convergence of $v_n$ then follows exactly as in the previous argument. Summing the bounds above yields the claimed rates. $\qed$
\end{appendix}

\end{document}